\documentclass[a4paper,10pt]{article}

\usepackage{amsmath,amssymb,amsthm,amscd}
\usepackage[mathscr]{eucal}
\usepackage{multirow}
\usepackage{hhline}
\usepackage[all]{xy}

\makeatletter
 
  \@addtoreset{equation}{subsection}
\makeatother

\newcommand{\plim}{\varprojlim}
\newcommand{\mcal}{\mathcal}

\newcommand{\mbb}{\mathbb}
\newcommand{\mrm}{\mathrm}
\newcommand{\mfS}{\mathfrak{S}}
\newcommand{\vphi}{\varphi}
\newcommand{\mfL}{\mathfrak{L}}
\newcommand{\mfM}{\mathfrak{M}}
\newcommand{\mfN}{\mathfrak{N}}

\newcommand{\mft}{\mathfrak{t}}
\newcommand{\mfm}{\mathfrak{m}}
\newcommand{\mfn}{\mathfrak{n}}
\newcommand{\ku}{k[\![u]\!]}
\newcommand{\wh}{\widehat}
\newcommand{\whR}{\widehat{\mathcal{R}}}

\newcommand{\Max}{\mathrm{Max}}

\newcommand{\wt}{\widetilde}
\newcommand{\Mod}{\mrm{Mod}}
\newcommand{\wtMod}{\widetilde{\mrm{Mod}}}

\newtheorem{theorem}{Theorem}[section]
\newtheorem{corollary}[theorem]{Corollary}
\newtheorem{lemma}[theorem]{Lemma}
\newtheorem{question}[theorem]{Question}
\newtheorem{proposition}[theorem]{Proposition}

\theoremstyle{definition}
\newtheorem{definition}[theorem]{Definition}
\newtheorem{remark}[theorem]{Remark}

\newcommand{\e}{\varepsilon}

\newcommand{\cO}{\mathcal{O}}

\newtheorem*{acknowledgements}{Acknowledgements}

\title{On Galois equivariance of homomorphisms between torsion potentially crystalline representations}

\author{Yoshiyasu Ozeki\footnote{
Research Institute for Mathematical Sciences, Kyoto University,
Kyoto 606-8502, JAPAN.
\endgraf
e-mail: {\tt yozeki@kurims.kyoto-u.ac.jp}
\endgraf
Supported by the JSPS Fellowships for Young Scientists.}}

\date{}
\begin{document}
\maketitle

\begin{abstract}
Let $K$ be a complete discrete valuation field of mixed characteristic $(0,p)$ 
with perfect residue field.
Let $(\pi_n)_{n\ge 0}$ be a system of $p$-power roots of a uniformizer 
$\pi=\pi_0$ of $K$ with $\pi^p_{n+1}=\pi_n$, 
and define $G_s$ (resp.\ $G_{\infty}$) the absolute Galois group of $K(\pi_s)$
(resp.\ $K_{\infty}:=\bigcup_{n\ge 0} K(\pi_n)$). 
In this paper, we study $G_s$-equivatiantness properties of 
$G_{\infty}$-equivariant homomorphisms 
between torsion (potentially) crystalline representations.
\end{abstract}

\tableofcontents


\section{Introduction}
Let $p$ be a prime number and $r\ge 0 $ an integer.
Let $K$ be a complete discrete valuation field 
of mixed characteristic $(0,p)$ with perfect residue field
and absolute ramification index $e$.
Let $\pi=\pi_0$ be a uniformizer of $K$ and 
$\pi_n$ a $p^n$-th root of $\pi$ such that $\pi^p_{n+1}=\pi_n$ for all $n\ge 0$.  
For any integer $s\ge 0$,
we put $K_{(s)}=K(\pi_s)$.
We also put $K_{\infty}=\bigcup_{n\ge 0}K_{(n)}$. 
We denote by $G_K, G_s$ and $G_{\infty}$ absolute Galois groups of 
$K$, $K_{(s)}$ and  $K_{\infty}$,
respectively.
By definition we have the following decreasing sequence of Galois groups:
$$
G_K=G_0\supset G_1\supset G_2\supset\cdots \supset G_{\infty}.  
$$
Since $K_{\infty}$ is a strict APF extension of $K$,
the theory of fields of norm implies that 
$G_{\infty}$ is isomorphic to the absolute Galois group of some field of characteristic $p$.
Therefore, representations of $G_{\infty}$ has easy interpretations
via Fontaine's \'etale $\vphi$-modules. 
Hence it seems natural to have the following question:
\begin{question}
\label{que1}
Let $T$ be a representation of $G_K$.
For a ``small'' integer $s\ge 0$, can we reconstruct various information of the $G_s$-action on $T$ 
from that of the $G_{\infty}$-action? 
\end{question}

\noindent
Nowadays there is an interesting insight of Breuil; 
he showed that representations of $G_K$ arising from finite flat group schemes or 
$p$-divisible groups over the integer ring of $K$ is ``determined'' by its restriction to $G_{\infty}$. 
Moreover, for $\mbb{Q}_p$-representations,
Kisin proved the following theorem in \cite{Kis}
(which was a conjecture of Breuil):
the restriction functor from the category of
crystalline $\mbb{Q}_p$-representations of $G_K$ 
into the category of $\mbb{Q}_p$-representations of $G_{\infty}$
is fully faithful. 

In this paper, we give some partial answers to Question \ref{que1} for 
{\it torsion crystalline representations}, moreover, {\it torsion potentially crystalline representations}.
Our first main result is as follows.
Let $\mrm{Rep}^{r,\mrm{ht},\mrm{pcris}(s)}_{\mrm{tor}}(G_K)$
be the category of torsion $\mbb{Z}_p$-representations $T$ of $G_K$
which satisfy the following:
there exist free $\mbb{Z}_p$-representations 
$L$ and $L'$ of $G_K$, of height $\le r$, such that
\begin{itemize}
\item  $L|_{G_s}$ is a subrepresentation of $L'|_{G_s}$.
Furthermore, $L|_{G_s}$ and $L'|_{G_s}$ are lattices in some 
crystalline $\mbb{Q}_p$-representation of $G_s$ with Hodge-Tate weights in $[0,r]$;
\item  $T|_{G_s} \simeq (L'|_{G_s})/(L|_{G_s})$.
\end{itemize}
\begin{theorem}
\label{Main1}
Suppose that $p$ is odd and $e(r-1)<p-1$.
Let $T$ and $T'$ be objects of $\mrm{Rep}^{r,\mrm{ht},\mrm{pcris}(s)}_{\mrm{tor}}(G_K)$.
Then any $G_{\infty}$-equivariant homomorphism $T\to T'$ is in fact $G_s$-equivariant.
\end{theorem}

\noindent
We should remark that 
the condition $e(r-1)<p-1$ in the above does not depend on $s$.
We put $\mrm{Rep}^{r, \mrm{cris}}_{\mrm{tor}}(G_K)=\mrm{Rep}^{r, \mrm{ht}, \mrm{pcris}(0)}_{\mrm{tor}}(G_K)$.
By definition,
a torsion $\mbb{Z}_p$-representation $T$ of $G_K$ 
is contained in this category
if and only if it can be written as the quotient of lattices in some crystalline $\mbb{Q}_p$-representation
of $G_K$ with Hodge-Tate weights in $[0,r]$. 
We call the objects in $\mrm{Rep}^{r, \mrm{cris}}_{\mrm{tor}}(G_K)$ 
{\it torsion crystalline representations
with Hodge-Tate weights in $[0,r]$}.
In the case $r=1$, such representations are equivalent to finite flat representations.
(Here, a torsion $\mbb{Z}_p$-representation of $G_K$ is finite flat
if it arises from the generic fiber of some $p$-power order 
finite flat commutative group scheme over the integer ring of $K$.)
Combining Theorem \ref{Main1} with results of  
\cite{Kim}, \cite{La}, \cite{Li3} (the case $p=2$)
we obtain the following full faithfulness theorem for torsion 
crystalline representations.

\begin{corollary}[Full Faithfulness Theorem]
\label{FFTHMtorcris}
Suppose $e(r-1)<p-1$.
Then the restriction functor
$\mrm{Rep}^{r, \mrm{cris}}_{\mrm{tor}}(G_K)\to \mrm{Rep}_{\mrm{tor}}(G_{\infty})$
is fully faithful.
\end{corollary}

\noindent
Before this work,
some results are already known.
First,
the full faithfulness theorem was proved by 
Breuil for  
$e=1$ and $r<p-1$ via the Fontaine-Laffaille theory (\cite{Br2}, the proof of Th\'eor\`em 5.2).
He also proved the theorem
under the assumptions $p>2$ and $r\le 1$ 
as a consequence of a study of the category of finite flat group schemes (\cite[Theorem 3.4.3]{Br3}).
Later, his result was extended to the case $p=2$ in \cite{Kim}, \cite{La}, \cite{Li3} 
(proved independently). 
In particular, the full faithfulness theorem for $p=2$ is 
a consequence of their works.
On the other hand, Abrashkin proved the full faithfulness 
in the case where $p>2, r<p$ and $K$ is a finite unramified extension of $\mbb{Q}_p$ (\cite[Section 8.3.3]{Ab2}).
His proof is based on calculations of ramification bounds for torsion crystalline representations. 
In \cite{Oz2}, a proof of Corollary \ref{FFTHMtorcris} under the assumption $er<p-1$ 
is given via $(\vphi,\hat{G})$-modules
(which was introduced by Tong Liu \cite{Li2} to classify lattices in 
semi-stable representations).

Our proof of Theorem \ref{Main1} is similar to the proof for the main result of \cite{Oz2},
but we need more careful considerations for $(\vphi,\hat{G})$-modules.
Indeed we need special base change arguments
to study some potential crystalline representations.
In fact, we prove a full faithfulness theorem for 
torsion representations arising from certain classes of 
$(\vphi,\hat{G})$-modules (cf.\ Theorem \ref{FFTHM}), 
which immediately gives our main theorem.
In addition,
our study gives a result as below which is the second main result of this paper
(here, we define $\mrm{log}_p(x):=-\infty$ for any real number $x\le 0$).

\begin{theorem}
\label{Main2}
Suppose that $p$ is odd and $s> n-1 + \mrm{log}_p(r-(p-1)/e)$.
Let $T$ and $T'$ be objects of $\mrm{Rep}^{r, \mrm{cris}}_{\mrm{tor}}(G_K)$
which are killed by $p^n$.
Then any $G_{\infty}$-equivariant homomorphism $T\to T'$ is in fact $G_s$-equivariant.
\end{theorem}

\noindent
For torsion semi-stable representations, a similar result was shown in 
Theorem 3 of \cite{CL2}, which was a consequence  of  a study of ramification bounds. 
The bound appearing in their theorem was $n-1 + \mrm{log}_p(nr)$. 
By applying our arguments given in this paper,
we can obtain a generalization of their result;
our refined condition is  $n-1 + \mrm{log}_pr$
(see Theorem \ref{Main3}).
Some other consequences of our study are described in subsection \ref{consequences}.
Motivated by the full faithfulness theorem 
(= Corollary \ref{FFTHMtorcris})
and Theorem \ref{Main2}, we raise the following question. 
\begin{question}
Is any $G_{\infty}$-equivariant homomorphism in the category $\mrm{Rep}^{r, \mrm{cris}}_{\mrm{tor}}(G_K)$
in fact $G_s$-equivariant
for $s>\mrm{log}_p(r-(p-1)/e)$?
\end{question}

On the other hand,
there exist counter examples of the full faithfulness theorem
when we ignore the condition
$e(r-1)< p-1$.
Let $\mrm{Rep}_{\mrm{tor}}(G_1)$ be the category of 
torsion $\mbb{Z}_p$-representations
of $G_1$.
\begin{theorem}[= Special case of Corollary \ref{nonfullthm}]
\label{nonfull}
Suppose that $K$ is a finite extension of $\mbb{Q}_p$,
and also suppose $e\mid (p-1)$ or $(p-1)\mid e$.
If $e(r-1)\ge p-1$,
the restriction functor 
$\mrm{Rep}^{r, \mrm{cris}}_{\mrm{tor}}(G_K)\to 
\mrm{Rep}_{\mrm{tor}}(G_1)$
is not full
$($in particular, the restriction functor 
$\mrm{Rep}^{r, \mrm{cris}}_{\mrm{tor}}(G_K)\to 
\mrm{Rep}_{\mrm{tor}}(G_{\infty})$ is not full$)$.
\end{theorem}
\noindent
In particular,
if $p=2$,
then the full faithfulness never hold for any 
finite extension $K$ of $\mbb{Q}_2$ and any $r\ge 2$.
Theorem \ref{nonfull} implies that the condition ``$e(r-1)<p-1$'' in
Corollary \ref{FFTHMtorcris} is the best possible 
for many finite extensions $K$ of $\mbb{Q}_p$.

Now we describe the organization of this paper.
In Section 2, we setup notations and summarize facts we need later.
In Section 3, we define variant notions of $(\vphi,\hat{G})$-modules and give some basic properties.
They are needed to study certain classes of potentially crystalline representations 
and restrictions of semi-stable representations.
In Section 4, we study technical torsion $(\vphi,\hat{G})$-modules
which are related with torsion (potentially) crystalline representations.
The key result in this section is the full faithfulness result Theorem \ref{FFTHM} on them, 
which allows us to prove our main results immediately.
Finally, in Section 5, we calculate the smallest integer $r$ 
for a given torsion representation $T$
such that $T$ admits a crystalline lift with Hodge-Tate weights in $[0,r]$. 
We mainly study the rank two case. 
We use our full faithfulness theorem to assure the non-existence of 
crystalline lifts with small Hodge-Tate weights.
Theorem \ref{nonfull} is a consequence of studies of this section.

\begin{acknowledgements}
The author would like to thank Shin Hattori, Naoki Imai and Yuichiro Taguchi 
who gave him many valuable advice.
This work was supported by JSPS KAKENHI Grant Number 25$\cdot$173.
\end{acknowledgements}

\noindent
{\bf Notation and convention:}
Throughout this paper, we fix a prime number $p$.
Except Section 5, we always assume that {\it $p$ is odd}.

For any topological group $H$,
we denote by 
$\mrm{Rep}_{\mrm{tor}}(H)$ 
(resp.\ $\mrm{Rep}_{\mbb{Z}_p}(H)$, resp.\ $\mrm{Rep}_{\mbb{Q}_p}(H)$)
the category of torsion $\mbb{Z}_p$-representations of $H$
(resp.\ the category of free $\mbb{Z}_p$-representations of $H$,
resp.\ the category of  $\mbb{Q}_p$-representations of $H$).
All $\mbb{Z}_p$-representations (resp.\ $\mbb{Q}_p$-representations) 
in this paper are always assumed to be finitely generated 
over $\mbb{Z}_p$ (resp.\ $\mbb{Q}_p$).

For any field $F$, we denote by $G_F$ the absolute Galois group of $F$
(for a fixed separable closure of $F$).


\section{Preliminaries}

In this section, 
we recall definitions and basic properties for  Kisin modules
and $(\vphi,\hat{G})$-modules.
Throughout Section 2, 3 and 4,
we always assume that $p$ is an odd prime.

\subsection{Basic notations}
Let $k$ be a perfect field of 
characteristic $p$,
$W(k)$ the ring of Witt vectors with coefficients in $k$, 
$K_0=W(k)[1/p]$, $K$ a finite totally 
ramified extension of $K_0$ of degree $e$,
$\overline{K}$ a fixed algebraic closure of $K$.
Throughout this paper,
we fix a uniformizer $\pi$ of $K$. 
Let $E(u)$ be the 
minimal polynomial of $\pi$ over $K_0$.
Then $E(u)$ is an Eisenstein polynomial.
For any integer $n\ge 0$,
we fix a system $(\pi_n)_{n\ge 0}$ of a $p^n$-th root of $\pi$ in $\overline{K}$
such that $\pi^p_{n+1}=\pi_n$.
Let $R=\plim \cO_{\overline{K}}/p$, 
where $\cO_{\overline{K}}$ is 
the integer ring of $\overline{K}$
and the transition maps are 
given by the $p$-th power map.
For any integer $s\ge 0$,
we write
$\underline{\pi_s}:=(\pi_{s+n})_{n\ge 0}\in R$ and 
$\underline{\pi}:=\underline{\pi_0}\in R$.
Note that we have $\underline{\pi_s}^{p^s}=\underline{\pi}$.

Let $L$ be the  completion of an unramified algebraic extension of $K$
with residue field $k_L$.
Then $\pi_s$ is a uniformizer of $L_{(s)}:=L(\pi_s)$ 
and $L_{(s)}$ is a totally ramified degree $ep^s$ extension of 
$L_0:=W(k_L)[1/p]$.
We set $L_{\infty}:=\bigcup_{n\ge 0}L_{(n)}$.
We put $G_{L,s}:=G_{L_{(s)}}=\mrm{Gal}(\overline{L}/L_{(s)})$ and
$G_{L,\infty}:=G_{L_{\infty}}=\mrm{Gal}(\overline{L}/L_{\infty})$.
By definitions,
we have $L=L_{(0)}$ and $G_{L,0}=G_L$.
Put $\mfS_{L,s}=W(k_L)[\![u_s]\!]$ (resp.\ $\mfS_L=W(k_L)[\![u]\!]$) 
with an indeterminate $u_s$ (resp.\ $u$). 
We equip 
a Frobenius endomorphism
$\varphi$ of $\mfS_{L,s}$ (resp.\ $\mfS_L$) 
by $u_s\mapsto u_s^p$ (resp.\ $u\mapsto u^p$) and 
the Frobenius on $W(k_L)$. 
We embed the $W(k_L)$-algebra $W(k_L)[u_s]$ (resp.\ $W(k_L)[u]$) into $W(R)$
via the map $u_s\mapsto [\underline{\pi_s}]$ 
(resp.\ $u\mapsto [\underline{\pi}]$),
where $[\ast]$ stands for the Teichm\"uller
representative.
This embedding extends to an 
embedding $\mfS_{L,s}\hookrightarrow W(R)$
(resp.\ $\mfS_L\hookrightarrow W(R)$).
By identifying $u$ with $u_s^{p^s}$,
we regard $\mfS_L$ as a subalgebra of $\mfS_{L,s}$.  
It is readily seen that the embedding 
$\mfS_L\hookrightarrow \mfS_{L,s} \hookrightarrow W(R)$ is 
compatible with the Frobenius endomorphisms.
If we denote by $E_s(u_s)$
the minimal polynomial of $\pi_s$ over $K_0$, with indeterminate $u_s$, 
then we have  $E_s(u_s)=E(u_s^{p^s})$.
Therefore, we have $E_s(u_s)=E(u)$ in $\mfS_{L,s}$.
We note that the minimal polynomial of $\pi_s$ over $L_0$
is $E_s(u_s)$.

Let $S^{\mrm{int}}_{L_0,s}$ (resp.\ $S^{\mrm{int}}_{L_0})$) be the $p$-adic completion 
of the divided power envelope of $W(k_L)[u_s]$ (resp.\ $W(k_L)[u]$) with respect to the ideal 
generated by $E_s(u_s)$ (resp.\ $E(u)$).  
There exists a unique Frobenius map $\vphi\colon S^{\mrm{int}}_{L_0,s}\to S^{\mrm{int}}_{L_0,s}$
(resp.\ $\vphi\colon S^{\mrm{int}}_{L_0}\to S^{\mrm{int}}_{L_0}$)
and monodromy $N\colon S^{\mrm{int}}_{L_0,s}\to S^{\mrm{int}}_{L_0,s}$ 
defined by $\varphi(u_s)=u_s^p$ (resp.\ $\varphi(u)=u^p$) 
and $N(u_s)=-u_s$ (resp.\ $N(u)=-u$).
Put $S_{L_0,s}=S^{\mrm{int}}_{L_0,s}[1/p]=L_0\otimes_{W(k_L)} S^{\mrm{int}}_{L_0,s}$
(resp.\ $S_{L_0}=S^{\mrm{int}}_{L_0}[1/p]=L_0\otimes_{W(k_L)} S^{\mrm{int}}_{L_0}$).
We equip $S^{\mrm{int}}_{L_0,s}$ and $S_{L_0,s}$ 
(resp.\ $S^{\mrm{int}}_{L_0}$ and $S_{L_0}$) with decreasing filtrations 
$\mrm{Fil}^iS^{\mrm{int}}_{L_0,s}$ and $\mrm{Fil}^iS_{L_0,s}$ 
(resp.\ $\mrm{Fil}^iS^{\mrm{int}}_{L_0,s}$ and $\mrm{Fil}^iS_{L_0,s}$)
by the $p$-adic completion of the ideal generated by $E^j_s(u_s)/j!$ (resp.\ $E^j(u)/j!$) 
for all $j\ge 0$. 
The inclusion $W(k_L)[u_s]\hookrightarrow W(R)$ (resp.\ $W(k_L)[u]\hookrightarrow W(R)$) 
via the map $u_s\mapsto [\underline{\pi_s}]$ (resp.\ $u\mapsto [\underline{\pi}]$)
induces $\vphi$-compatible inclusions 
$\mfS_{L,s}\hookrightarrow S^{\mrm{int}}_{L_0,s}\hookrightarrow A_{\mrm{cris}}$
and $S_{L_0,s}\hookrightarrow B^+_{\mrm{cris}}$
(resp.\ $\mfS_L\hookrightarrow S^{\mrm{int}}_{L_0}\hookrightarrow A_{\mrm{cris}}$
and $S_{L_0}\hookrightarrow B^+_{\mrm{cris}}$).
By these inclusions,
we often regard these rings as subrings of $B^+_{\mrm{cris}}$. 
By identifying $u$ with $u_s^{p^s}$ as before,
we regard $S^{\mrm{int}}_{L_0}$  (resp.\ $S_{L_0}$)
as a $\vphi$-stable (but not $N$-stable) subalgebra of 
$S^{\mrm{int}}_{L_0,s}$ (resp.\ $S_{L_0,s}$).  
By definitions,
we have $\mfS_L=\mfS_{L,0},\ S^{\mrm{int}}_{L_0,0}=S^{\mrm{int}}_{L_0}$
and $S_{L_0,0}= S_{L_0}$.\\

\noindent
{\bf Convention:}
For simplicity, 
if $L=K$, then we often omit the subscript ``$L$'' from various notations 
(e.g. $G_{K_s}=G_s$,
$G_{K_{\infty}}=G_{\infty}$,
$\mfS_K=\mfS, \mfS_{K,s}=\mfS_s$).\\

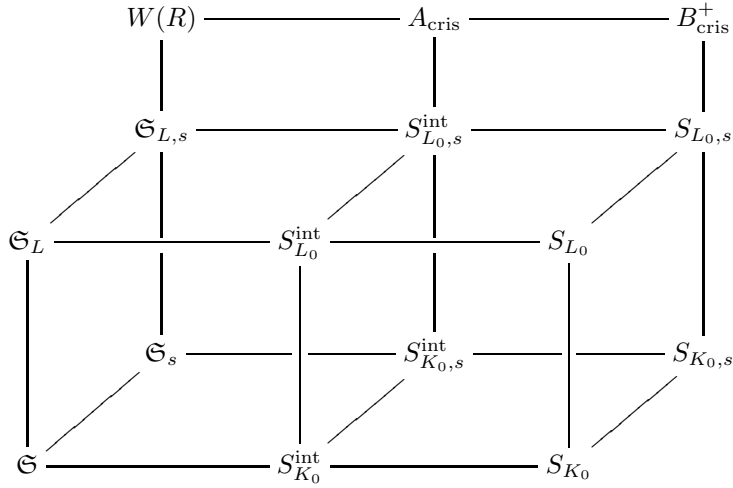
\begin{figure}[htbp]
\begin{center}
$\displaystyle \xymatrix{
& W(R) \ar@{-}[rr]
& 
& A_{\mrm{cris}} \ar@{-}[rr]
& 
& B^+_{\mrm{cris}}
\\
&
\mfS_{L,s} \ar@{-}[rr] \ar@{-}[u]
& 
& S^{\mrm{int}}_{L_0,s}\ar@{-}[rr] \ar@{-}[u]
& 
& S_{L_0,s} \ar@{-}[u]
\\
\mfS_L \ar@{-}[rr] \ar@{-}[ru]
& \ar@{-}[u]
& S^{\mrm{int}}_{L_0}\ar@{-}[rr] \ar@{-}[ru]
& \ar@{-}[u]
& S_{L_0} \ar@{-}[ru]
& 
\\
&
\mfS_s \ar@{-}[r] \ar@{-}[u]
& \ar@{-}[r]
& S^{\mrm{int}}_{K_0,s}\ar@{-}[r] \ar@{-}[u]
& \ar@{-}[r]
& S_{K_0,s} \ar@{-}[uu]
\\
\mfS \ar@{-}[rr] \ar@{-}[uu] \ar@{-}[ru]
& 
& S^{\mrm{int}}_{K_0}\ar@{-}[rr] \ar@{-}[uu] \ar@{-}[ru]
& 
& S_{K_0}  \ar@{-}[uu] \ar@{-}[ru]
&
}$
\end{center}
\caption{Ring extensions}
\end{figure}


\subsection{Kisin modules}
Let $r, s\ge 0$ be  integers.
A {\it $\vphi$-module} {\it over $\mfS_{L,s}$}
is an $\mfS_{L,s}$-module 
$\mfM$ equipped with a $\vphi$-semilinear map 
$\vphi\colon \mfM\to \mfM$.
A morphism between two $\vphi$-modules 
$(\mfM_1,\vphi_1)$ and $(\mfM_2,\vphi_2)$ over $\mfS_{L,s}$
is an $\mfS_{L,s}$-linear map $\mfM_1\to \mfM_2$
compatible with $\vphi_1$ and $\vphi_2$.
Denote by $'\mrm{Mod}^r_{/\mfS_{L,s}}$
the category of $\vphi$-modules $(\mfM,\vphi)$ over $\mfS_{L,s}$
{\it of height $\le r$}  
in the sense that $\mfM$ is of finite type 
over $\mfS_{L,s}$ and the cokernel of 
$1\otimes \vphi\colon \mfS_{L,s}\otimes_{\vphi,\mfS_{L,s}}\mfM\to \mfM$
is killed by $E_s(u_s)^r$.

Let $\mrm{Mod}^r_{/\mfS_{L,s}}$
be the full subcategory of $'\mrm{Mod}^r_{/\mfS_{L,s}}$
consisting of finite free $\mfS_{L,s}$-modules.
We call an object of $\mrm{Mod}^r_{/\mfS_{L,s}}$ 
a {\it free Kisin module} {\it of height $\le r$ $($over $\mfS_{L,s})$}.

Let $\mrm{Mod}^r_{/\mfS_{L,s,\infty}}$
be the full subcategory of $'\mrm{Mod}^r_{/\mfS_{L,s}}$
consisting of finite $\mfS_{L,s}$-modules 
which are killed by some power of $p$ and have projective dimension $1$ 
in the sense that $\mfM$ has a two term resolution  by
finite free $\mfS_{L,s}$-modules.
We call an object of $\mrm{Mod}^r_{/\mfS_{L,s,\infty}}$ 
a {\it torsion Kisin module of height $\le r$ $($over $\mfS_{L,s})$}.

For any free or torsion Kisin module $\mfM$ over $\mfS_{L,s}$,
we define a $\mbb{Z}_p[G_{L,\infty}]$-module  $T_{\mfS_{L,s}}(\mfM)$
by
\begin{align*}
T_{\mfS_{L,s}}(\mfM):=
\left\{
\begin{array}{ll}
\mrm{Hom}_{\mfS_{L,s},\vphi}(\mfM,W(R))\hspace{21mm} {\rm if}\ \mfM\ {\rm is}\ {\rm free},  
\cr
\mrm{Hom}_{\mfS_{L,s},\vphi}(\mfM,\mbb{Q}_p/\mbb{Z}_p\otimes_{\mbb{Z}_p} W(R))\quad {\rm if}\ \mfM\ {\rm is}\ {\rm torsion}. 
\end{array}
\right.
\end{align*}
Here a $G_{L,\infty}$-action on 
$T_{\mfS_{L,s}}(\mfM)$ is given by 
$(\sigma.g)(x)=\sigma(g(x))$ 
for $\sigma\in G_{L,\infty}, g\in T_{\mfS}(\mfM), x\in \mfM$.\\

\noindent
{\bf Convention:}
For simplicity, 
if $L=K$, then we often omit the subscript ``$L$'' from various notations 
(e.g. $\mrm{Mod}^r_{/\mfS_{K,s,\infty}}=\mrm{Mod}^r_{/\mfS_{s,\infty}}$,
$T_{\mfS_{K,s}}=T_{\mfS_s}$
).
Also, if $s=0$, we often omit the subscript ``$s$'' from various notations 
(e.g. $\mrm{Mod}^r_{/\mfS_{L,0,\infty}}=\mrm{Mod}^r_{/\mfS_{L,\infty}}$,
$T_{\mfS_{L,0}}=T_{\mfS_L}$,
$\mrm{Mod}^r_{/\mfS_{K,0,\infty}}=\mrm{Mod}^r_{/\mfS_{\infty}}$,
$T_{\mfS_{K,0}}=T_{\mfS}$
).\\

\begin{proposition}
\label{Kisinfunctor}
$(1)$ {\rm (\cite[Corollary 2.1.4 and Proposition 2.1.12]{Kis})}\ 
The functor $T_{\mfS_{L,s}}\colon \mrm{Mod}^r_{/\mfS_{L,s}}\to 
\mrm{Rep}_{\mbb{Z}_p}(G_{\infty})$
is exact and fully faithful.

\noindent
$(2)$ {\rm (\cite[Corollary 2.1.6, 3.3.10 and 3.3.15]{CL1})}\ 
The functor $T_{\mfS_{L,s}}\colon \mrm{Mod}^r_{/\mfS_{L,s,\infty}}\to 
\mrm{Rep}_{\mrm{tor}}(G_{\infty})$
is exact and faithful.
Furthermore, it is full if $er<p-1$.
\end{proposition}


\subsection{$(\vphi,\hat{G})$-modules}
\label{Liumodule:section}
The notion of $(\vphi,\hat{G})$-modules are introduced by Tong Liu in \cite{Li2}
to classify lattices in semi-stable representations. 
We recall definitions and properties of them.
We continue to use same notations as above.

Let $L_{p^{\infty}}$ be the field obtained by 
adjoining all $p$-power roots of unity to $L$.
We denote by $\hat{L}$ the composite field of $L_{\infty}$ and $L_{p^{\infty}}$.
We define 
$H_L:=\mrm{Gal}(\hat{L}/L_{\infty})$, $H_{L,\infty}:=\mrm{Gal}(\overline{K}/\hat{L})$
$G_{L,p^{\infty}}:=\mrm{Gal}(\hat{L}/L_{p^{\infty}})$ and $\hat{G}_L:=\mrm{Gal}(\hat{L}/L)$.
Furthermore, putting  $L_{(s),p^{\infty}}=L_{(s)}L_{p^{\infty}}$,
we define $\hat{G}_{L,s}=\mrm{Gal}(\hat{L}/L_{(s)})$ and
$G_{L,s,p^{\infty}}:=\mrm{Gal}(\hat{L}/L_{(s), p^{\infty}})$.

\begin{figure}[htbp]
\begin{center}
\[
\xymatrix{
& & & & & \bar{K} \\
& & & & & \hat{L} \ar@{-}[u] \ar@{-}[u] \ar@/_1pc/@{-}[u]_{H_{L,\infty}}
\\
& & & L_{p^{\infty}} 
\ar@/^1pc/@{-}[rru] ^{G_{L,p^{\infty}}}
\ar@{-}[rru] 
& & & & \\
& & & & & 
L_{\infty} \ar@/_1pc/@{-}[uu] _{H_L}
\ar@{-}[uu]
\ar@/_4pc/@{-}[uuu] _{G_{L,\infty}}
& & & & \\  
& & & L \ar@{-}[rru]  
\ar@/^1pc/@{-}[rruuu] ^{\hat{G}_L}
\ar@{-}[rruuu] 
\ar@/^6pc/@{-}[rruuuu] ^{G_L}
\ar@{-}[uu] 
& & & & \\ 
}
\]
\end{center}
\caption{Galois groups of field extensions}
\end{figure}
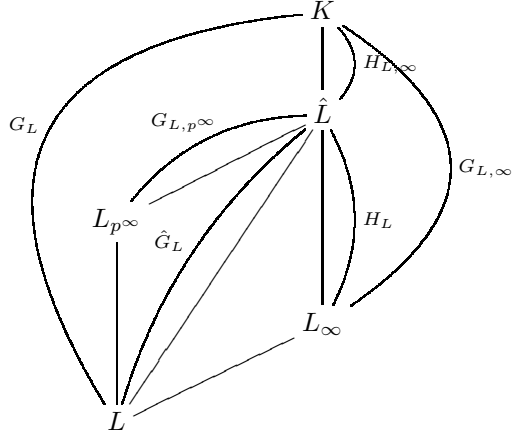

Since $p>2$,
it is known that 
$L_{(s),p^{\infty}}\cap L_{\infty}=L_{(s)}$ and thus 
$\hat{G}_{L,s}=G_{L,s,p^{\infty}}\rtimes H_{L,s}$.
Furthermore, $G_{L,s,p^{\infty}}$ is topologically isomorphic to $\mbb{Z}_p$.

\begin{lemma}
\label{easylemma}
A natural map $G_{L,s,p^{\infty}}\to G_{K,s,p^{\infty}}$ defined by $g\mapsto g|_{\hat{K}}$ is bijective. 
\end{lemma}
\begin{proof}
By replacing $L_s$ with $L$,
we may assume $s=0$.
It suffices to prove $\hat{K}\cap L_{p^{\infty}}=K_{p^{\infty}}$.
Since $G_{K,p^{\infty}}$ is isomorphic to $\mbb{Z}_p$, 
we know that any finite subextension of $\hat{K}/K_{p^{\infty}}$ 
is of the form $K_{(s),p^{\infty}}$
for some $s\ge 0$.
Assume that we have $\hat{K}\cap L_{p^{\infty}}\not=K_{p^{\infty}}$.
Then we have $K_{(1)}\subset \hat{K}\cap L_{p^{\infty}}\subset L_{p^{\infty}}$.
Thus $\pi_1$ is contained in $L_{p^{\infty}}\cap L_{\infty}=L$.
However, since $L$ is unramified over $K$, this contradicts the fact that 
$\pi$ is a uniformizer of $L$.
\end{proof}

We fix a topological generator $\tau$ of $G_{K,p^{\infty}}$.
We also denote by  $\tau$ 
the pre-image of $\tau\in G_{K,p^{\infty}}$ for the bijection 
$G_{L,p^{\infty}}\simeq G_{K,p^{\infty}}$ of the above lemma.
Note that $\tau^{p^s}$ is a topological generator of $G_{L,s,p^{\infty}}$.

For any $g\in G_K$, we put $\underline{\e}(g)=g(\underline{\pi})/\underline{\pi}\in R$,
and define $\underline{\e}:=\underline{\e}(\tilde{\tau})$. 
Here, $\tilde{\tau}\in G_K$ is any lift of $\tau\in \hat{G}_K$ 
and then $\underline{\e}(\tilde{\tau})$ is independent of the choice of the lift of $\tau$.
With these notation, we also note that
we have $g(u)=[\underline{\e}(g)]u$ (recall that $\mfS$ is embedded in $W(R)$). 
An easy computation shows that 
$\tau(\underline{\pi})/\underline{\pi}=\tau^{p^s}(\underline{\pi_s})/\underline{\pi_s}=\underline{\e}$.
Therefore, we have $\tau(u)/u=\tau^{p^s}(u_s)/u_s=[\underline{\e}]$.

We put $t=-\mrm{log}([\underline{\e}])\in A_{\mrm{cris}}$.
Denote by $\nu\colon W(R)\to W(\overline{k})$ 
a unique lift of the projection $R\to \overline{k}$,
which extends to a map 
$\nu \colon B^+_{\mrm{cris}}\to W(\overline{k})[1/p]$.
For any subring $A\subset B^+_{\mrm{cris}}$,
we put 
$I_+A=\mrm{Ker}(\nu\ \mrm{on}\  B^+_{\mrm{cris}})\cap A$.
For any integer $n\ge 0$,
let $t^{\{n\}}:=t^{r(n)}\gamma_{\tilde{q}(n)}(\frac{t^{p-1}}{p})$ 
where $n=(p-1)\tilde{q}(n)+r(n)$ with $\tilde{q}(n)\ge 0,\ 0\le r(n) <p-1$
and $\gamma_i(x)=\frac{x^i}{i!}$ is 
the standard divided power.
We define a subring $\mcal{R}_{L_0,s}$ 
(resp.\ $\mcal{R}_{L_0}$)
of $B^+_{\mrm{cris}}$ 
as below:
$$
\mcal{R}_{L_0,s}:=\{\sum^{\infty}_{i=0} f_it^{\{i\}}\mid f_i\in S_{L_0,s}\
\mrm{and}\ f_i\to 0\ \mrm{as}\ i\to \infty\}
$$
$$
({\rm resp}.\quad  \mcal{R}_{L_0}:=\{\sum^{\infty}_{i=0} f_it^{\{i\}}\mid f_i\in S_{L_0}\
\mrm{and}\ f_i\to 0\ \mrm{as}\ i\to \infty\}).
$$

\noindent
Put $\wh{\mcal{R}}_{L,s}=\mcal{R}_{L_0,s}\cap W(R)$
(resp.\ $\wh{\mcal{R}}_{L}=\mcal{R}_{L_0}\cap W(R)$)
and $I_{+,L,s}=I_+\wh{\mcal{R}}_{L,s}$
(resp.\ $I_{+,L}=I_+\wh{\mcal{R}}_L$).
By definitions, we have 
$\mcal{R}_{L_0,0}=\mcal{R}_{L_0}$,
$\wh{\mcal{R}}_{L,0}=\wh{\mcal{R}}_{L}$ and 
$I_{+,L,0}=I_{+,L}$.
Lemma 2.2.1 in \cite{Li2} shows that $\wh{\mcal{R}}_{L,s}$ 
$($resp.\ $\mcal{R}_{L_0,s})$ 
is a $\vphi$-stable $\mfS_{L,s}$-subalgebra of $W(R)$ $($resp.\ $B^+_{\mrm{cris}})$, and $\nu$ induces 
$\mcal{R}_{L_0,s}/I_+\mcal{R}_{L_0,s}\simeq L_0$ and 
$\wh{\mcal{R}}_{L,s}/I_{+,L,s}\simeq S^{\mrm{int}}_{L_0,s}/I_+S^{\mrm{int}}_{L_0,s}
\simeq \mfS_{L,s}/I_+\mfS_{L,s}\simeq W(k_L)$.
Furthermore, $\wh{\mcal{R}}_{L,s}, I_{+,L,s}, \mcal{R}_{L_0,s}$ 
and $I_+\mcal{R}_{L_0,s}$ are $G_{L,s}$-stable, and 
$G_{L,s}$-actions on them
factors through $\hat{G}_{L,s}$.
For any torsion Kisin module $\mfM$ over $\mfS_{L,s}$,
we equip $\whR_{L,s}\otimes_{\vphi,\mfS_{L,s}} \mfM$ with 
a Frobenius by
$\vphi_{\wh{\mcal{R}}_{L,s}}\otimes \vphi_{\mfM}$.
It is known that the natural map
$
\mfM\rightarrow \wh{\mcal{R}}_{L,s}\otimes_{\vphi, \mfS_{L,s}} \mfM
$
given by $x\mapsto 1\otimes x$ is an injection (cf.\ \cite[Corollary 2.12]{Oz1}).
By this injection, 
we regard $\mfM$ as a $\vphi(\mfS_{L,s})$-stable 
submodule of $\whR_{L,s}\otimes_{\vphi,\mfS_{L,s}} \mfM$.

\begin{definition}
\label{Liumod}
A {\it free} (resp.\ {\it torsion}) {\it $(\vphi, \hat{G}_{L,s})$-module of height} $\le r$ 
over $\mfS_{L,s}$
is a triple $\hat{\mfM}=(\mfM, \vphi_{\mfM}, \hat{G}_{L,s})$ where
\begin{enumerate}
\item[(1)] $(\mfM, \vphi_{\mfM})$ is a free (resp.\ torsion)
           Kisin module of height $\le r$ over $\mfS_{L,s}$, 

\vspace{-2mm}
           
\item[(2)] $\hat{G}_{L,s}$ is an $\wh{\mcal{R}}_{L,s}$-semilinear
           $\hat{G}_{L,s}$-action on $\whR_{L,s}\otimes_{\vphi, \mfS_{L,s}} \mfM$  
           which induces a continuous $G_{L,s}$-action on 
           $W(\mrm{Fr}R)\otimes_{\vphi, \mfS_{L,s}} \mfM$
           for the weak topology\footnote{This continuity condition is needed to assure that the $G_{L,s}$-action on 
           $\hat{T}_{L,s}(\hat{\mfM})$, defined after Definition \ref{Liumod},  
           is continuous (here, note that $\hat{T}_{L,s}(\hat{\mfM})$
           is the dual of $(W(\mrm{Fr}R)\otimes_{\vphi, \mfS_{L,s}} \mfM)^{\vphi=1}$;
           see Corollary 3.12 of \cite{Oz1}).},
           
\vspace{-2mm}           
           
\item[(3)] the $\hat{G}_{L,s}$-action commutes with $\vphi_{\wh{\mcal{R}}_{L,s}}\otimes \vphi_{\mfM}$,

\vspace{-2mm}

\item[(4)] $\mfM\subset (\whR_{L,s}\otimes_{\vphi,\mfS_{L,s}} \mfM)^{H_L}$,

\vspace{-2mm}

\item[(5)] $\hat{G}_{L,s}$ acts on the 
           $W(k_L)$-module $(\whR_{L,s}\otimes_{\vphi,\mfS_{L,s}} \mfM)/
            I_{+,L,s}(\whR_{L,s}\otimes_{\vphi,\mfS_{L,s}} \mfM)$ 
           trivially.           
\end{enumerate}
\end{definition}

A morphism 
between two $(\vphi, \hat{G}_{L,s})$-modules $\hat{\mfM}_1=(\mfM_1, \vphi_1, \hat{G})$ 
and $\hat{\mfM}_2=(\mfM_2, \vphi_2, \hat{G})$ is a morphism 
$f\colon \mfM_1\to \mfM_2$ of $\vphi$-modules over $\mfS_{L,s}$
such that 
$\wh{\mcal{R}}_{L,s}\otimes f\colon \whR_{L,s}\otimes_{\vphi,\mfS_{L,s}} \mfM_1\to 
\whR_{L,s}\otimes_{\vphi,\mfS_{L,s}} \mfM_2$
is $\hat{G}_{L,s}$-equivariant.
We denote by $\mrm{Mod}^{r,\hat{G}_{L,s}}_{/\mfS_{L,s}}$ 
(resp.\ $\mrm{Mod}^{r,\hat{G}_{L,s}}_{/\mfS_{L,s,\infty}}$)
the 
category of free (resp.\ torsion) $(\vphi, \hat{G}_{L,s})$-modules of height 
$\le r$ over $\mfS_{L,s}$.
We often regard $\whR_{L,s}\otimes_{\vphi,\mfS_{L,s}} \mfM$ 
as a $G_{L,s}$-module 
via the projection $G_{L,s}\twoheadrightarrow \hat{G}_{L,s}$.

For any free or torsion $(\vphi, \hat{G}_{L,s})$-module $\hat{\mfM}$ over $\mfS_{L,s}$,
we define a $\mbb{Z}_p[G_{L,s}]$-module  $\hat{T}_{L,s}(\hat{\mfM})$
by
\begin{align*}
\hat{T}_{L,s}(\hat{\mfM})=
\left\{
\begin{array}{ll}
\mrm{Hom}_{\wh{\mcal{R}}_{L,s},\vphi}
(\whR_{L,s}\otimes_{\vphi,\mfS_{L,s}} \mfM, W(R))
\hspace{21mm} {\rm if}\ \mfM\ {\rm is}\ {\rm free},  
\cr
\mrm{Hom}_{\wh{\mcal{R}}_{L,s},\vphi}
(\whR_{L,s}\otimes_{\vphi,\mfS_{L,s}} \mfM, \mbb{Q}_p/\mbb{Z}_p\otimes_{\mbb{Z}_p}W(R))
\hspace{3.5mm} {\rm if}\ \mfM\ {\rm is}\ {\rm torsion}.  
\end{array}
\right.
\end{align*}
Here,  $G_{L,s}$ 
acts on $\hat{T}_{L,s}(\hat{\mfM})$ by $(\sigma.f)(x)=\sigma(f(\sigma^{-1}(x)))$
for $\sigma\in G_{L,s},\ f\in \hat{T}_{L,s}(\hat{\mfM}),\ 
x\in \whR_{L,s}\otimes_{\vphi,\mfS_{L,s}}\mfM$. \\
Then, there exists a natural $G_{L,\infty}$-equivariant map 
\[
\theta_{L,s}\colon T_{\mfS_{L,s}}(\mfM)\to \hat{T}_{L,s}(\hat{\mfM})
\]
defined by 
$\theta(f)(a\otimes x)=a\vphi(f(x))$ for 
$f\in T_{\mfS_{L,s}}(\mfM),\ a\in \wh{\mcal{R}}_{L,s}, x\in \mfM$.
We have

\begin{theorem}[{\cite[Theorem 2.3.1 (1)]{Li2},\ \cite[Theorem 3.1.3 (1)]{CL2}}] 
The map $\theta_{L,s}$ is an isomorphism of $\mbb{Z}_p[G_{L,\infty}]$-modules.
\end{theorem}

\noindent
{\bf Convention:}
For simplicity, 
if $L=K$, then we often omit the subscript ``$L$'' from various notations 
(e.g. ``a $(\vphi, \hat{G}_{K,s})$-module'' = ``a $(\vphi, \hat{G}_s)$-module'',
      $\mrm{Mod}^{r,\hat{G}_{K,s}}_{/\mfS_{K,s}}
       =\mrm{Mod}^{r,\hat{G}_s}_{/\mfS_s}$,
      $\mrm{Mod}^{r,\hat{G}_{K,s}}_{/\mfS_{K,s,\infty}}
       =\mrm{Mod}^{r,\hat{G}_s}_{/\mfS_{s,\infty}}$,
      $\hat{T}_{K,s}=\hat{T}_s$, $\theta_{K,s}=\theta_s$).
Furthermore, if $s=0$, we often omit the subscript ``$s$'' from various notations 
(e.g. $\mrm{Mod}^{r,\hat{G}_{L,0}}_{/\mfS_{L,0}}
       =\mrm{Mod}^{r,\hat{G}_L}_{/\mfS_L}$,
      $\mrm{Mod}^{r,\hat{G}_{L,0}}_{/\mfS_{L,0,\infty}}
       =\mrm{Mod}^{r,\hat{G}_L}_{/\mfS_{L,\infty}}$,
      $\hat{T}_{L,0}=\hat{T}_L$, 
       $\mrm{Mod}^{r,\hat{G}_{K,0}}_{/\mfS_{K,0}}
       =\mrm{Mod}^{r,\hat{G}}_{/\mfS}$,
       ``a $(\vphi, \hat{G}_{K,0})$-module'' = ``a $(\vphi, \hat{G})$-module'',
      $\hat{T}_{K,0}=\hat{T}$, $\theta_{K,0}=\theta$).\\

Let $\mrm{Rep}^{r,\mrm{st}}_{\mbb{Q}_p}(G_{L,s})$
    (resp.\ $\mrm{Rep}^{r,\mrm{cris}}_{\mbb{Q}_p}(G_{L,s})$,
     resp.\ $\mrm{Rep}^{r,\mrm{st}}_{\mbb{Z}_p}(G_{L,s})$, 
     resp.\ $\mrm{Rep}^{r,\mrm{cris}}_{\mbb{Z}_p}(G_{L,s})$) 
be the categories of semi-stable $\mbb{Q}_p$-representations of $G_{L,s}$
     with Hodge-Tate weights in $[0,r]$
    (resp.\ the categories of crystalline $\mbb{Q}_p$-representations of $G_{L,s}$
     with Hodge-Tate weights in $[0,r]$,
     resp.\ the categories of lattices in semi-stable $\mbb{Q}_p$-representations of $G_{L,s}$
     with Hodge-Tate weights in $[0,r]$, 
     resp.\ the categories of lattices in crystalline $\mbb{Q}_p$-representations of $G_{L,s}$
     with Hodge-Tate weights in $[0,r]$).

There exists $\mft\in W(R)\smallsetminus pW(R)$ such that 
$\vphi(\mft)=pE(0)^{-1}E(u)\mft$. 
Such $\mft$ is unique up to units of $\mbb{Z}_p$ (cf.\ \cite[Example 2.3.5]{Li2}).  
Now we define the full subcategory 
$\mrm{Mod}^{r,\hat{G},\mrm{cris}}_{/\mfS}$ 
(resp.\ $\mrm{Mod}^{r,\hat{G},\mrm{cris}}_{/\mfS_{\infty}}$) 
of $\mrm{Mod}^{r,\hat{G}}_{/\mfS}$
(resp.\ $\mrm{Mod}^{r,\hat{G}}_{/\mfS_{\infty}}$)  
consisting of objects $\hat{\mfM}$ which satisfy the following condition; 
$
\tau(x)-x\in u^p\vphi(\mft)(W(R)\otimes_{\vphi,\mfS} \mfM)
$
for any $x\in \mfM$.

The following results are  important properties for the functor $\hat{T}_{L,s}$.

\begin{theorem}
\label{Thm1}
\noindent
$(1)$ {\rm (\cite[Theorem 2.3.1 (2)]{Li2})}\
The functor $\hat{T}$ induces an anti-equivalence of categories 
between $\mrm{Mod}^{r,\hat{G}}_{/\mfS}$ and $\mrm{Rep}^{r,\mrm{st}}_{\mbb{Z}_p}(G_K)$.

\noindent
$(2)$ {\rm (\cite[Proposition 5.9]{GLS},\ \cite[Theorem 19]{Oz2})}\ 
The functor $\hat{T}$ induces an anti-equivalence of categories 
between $\mrm{Mod}^{r,\hat{G},\mrm{cris}}_{/\mfS}$ and 
$\mrm{Rep}^{r,\mrm{cris}}_{\mbb{Z}_p}(G_K)$.

\noindent
$(3)$ {\rm (\cite[Corollary 2.8 and 5.34]{Oz1})}\
The functor 
$\hat{T}_{L,s}\colon \mrm{Mod}^{r,\hat{G}_{L,s}}_{/\mfS_{L,s,\infty}}
\to \mrm{Rep}_{\mrm{tor}}(G_{L,s})$ is exact and faithful. 
Furthermore, it is full if $er<p-1$.
\end{theorem}


\subsection{$(\vphi, \hat{G})$-modules, Breuil modules and filtered $(\vphi,N)$-modules}
\label{relations}
We recall some relations between Breuil modules and $(\vphi,\hat{G})$-modules.
Here we give a rough sketch only.
For more precise information, see \cite[Section 6]{Br1}, \cite[Section 5]{Li1} and the proof of \cite[Theorem 2.3.1 (2)]{Li2}.

Let $\hat{\mfM}$ be a free $(\vphi,\hat{G}_{L,s})$-module over $\mfS_{L,s}$.
If we put $\mcal{D}:=S_{L_0,s}\otimes_{\vphi,\mfS_{L,s}} \mfM$,
then $\mcal{D}$ has a structure of a Breuil module over $S_{L_0,s}$
which corresponds to the semi-stable representation $\mbb{Q}_p\otimes_{\mbb{Z}_p} \hat{T}_{L,s}(\hat{\mfM})$
of $G_{L,s}$
(for the definition and properties of Breuil modules, see \cite{Br1}).
Thus $\mcal{D}$ is equipped with a Frobenius $\vphi_{\mcal{D}}(=\vphi_{S_{L_0,s}}\otimes \vphi_{\mfM})$,
a decreasing filtration $(\mrm{Fil}^i\mcal{D})_{i\ge 0}$ 
of $S_{L_0,s}$-submodules of $\mcal{D}$ and a $L_0$-linear monodromy operator 
$N\colon \mcal{D}\to \mcal{D}$ which satisfy certain properties (for example, Griffiths transversality).

Putting $D=\mcal{D}/I_+S_{L_0,s}\mcal{D}$,
we can associate a filtered $(\vphi,N)$-module over $L_{(s)}$ as following:
$\vphi_{D}:=\vphi_{\mcal{D}}\ \mrm{mod}\ I_+S_{L_0,s}\mcal{D}$,
$N_D:= N_{\mcal{D}}\ \mrm{mod}\ I_+S_{L_0,s}\mcal{D}$ and 
$\mrm{Fil}^iD_{L_{(s)}}:=f_{\pi_s}(Fil^i(\mcal{D}))$.
Here, $f_{\pi_s}\colon \mcal{D}\to D_{L_{(s)}}$ 
is the projection defined  by
$\mcal{D}\twoheadrightarrow \mcal{D}/\mrm{Fil}^1S_{L_0,s}\mcal{D}\simeq D_{L_{(s)}}$.
Proposition 6.2.1.1 of \cite{Br1} implies that there 
exists a unique $\vphi$-compatible section $s\colon D\hookrightarrow \mcal{D}$ of 
$\mcal{D}\twoheadrightarrow D$.
By this embedding, we regard $D$ as a submodule of $\mcal{D}$. 
Then we have $N_{\mcal{D}}|_{D}=N_D$
and $N_{\mcal{D}}=N_{S_{L_0,s}}\otimes \mrm{Id}_D + \mrm{Id}_{S_{L_0,s}}\otimes N_D$
under the identification $\mcal{D}=S_{L_0,s} \otimes_{L_{(s)}} D$.

The $G_{L,s}$-action on $\whR_{L,s}\otimes_{\vphi,\mfS_{L,s}}\mfM$
extends to 
$B^+_{\mrm{cris}}\otimes_{\whR_{L,s}} (\whR_{L,s}\otimes_{\vphi,\mfS_{L,s}}\mfM)
\simeq B^+_{\mrm{cris}}\otimes_{S_{L_0,s}}\mcal{D}$.
This action is in fact explicitly written as follows:
\begin{equation}
\label{explicit}
g(a\otimes x)=\sum^{\infty}_{i=0}g(a)\gamma_i
(-\mrm{log}(\frac{g[\underline{\pi_s}]}{[\underline{\pi_s}]}))\otimes N^i_{\mcal{D}}(x)\quad
{\rm for}\ g\in G_{L,s}, a\in B^+_{\mrm{cris}}, x\in \mcal{D}.
\end{equation}
By this explicit formula,
we can obtain an easy relation between $N_{\mcal{D}}$ and $\tau^{p^s}$-action on $\hat{\mfM}$ as follows:
first we recall that 
$t=-\mrm{log}(\tau([\underline{\pi}])/[\underline{\pi}])
  =-\mrm{log}(\tau^{p^s}([\underline{\pi_s}])/[\underline{\pi_s}])$.
By the formula, for any $n\ge 0$ and $x\in \mcal{D}$, an induction on $n$ shows that
we have 
$$
(\tau^{p^s}-1)^n(x)=\sum^{\infty}_{m=n}(\sum_{i_1+\cdots i_n=m, i_j\ge 0}\frac{m!}{i_1!\cdots i_n!})
\gamma_m(t)\otimes N^m_{\mcal{D}}(x)
\in B^{+}_{\mrm{cris}}\otimes_{S_{L_0,s}} \mcal{D}
$$
and in particular we see $\frac{(\tau^{p^s}-1)^n}{n}(x)\to 0$ $p$-adically as $n\to \infty$.
Hence we can define 
\[
\mrm{log}(\tau^{p^s})(x):=
\sum^{\infty}_{n=1}(-1)^{n-1}\frac{(\tau^{p^s}-1)^n}{n}(x) \in B^{+}_{\mrm{cris}}\otimes_{S_{L_0,s}} \mcal{D}.
\]
It is not difficult to check the equation 
\begin{equation}
\label{eq1}
\mrm{log}(\tau^{p^s})(x)=t\otimes N_{\mcal{D}}(x).
\end{equation}


\subsection{Base changes for Kisin modules}
Let $\mfM$ be a free or torsion Kisin module of height $\le r$ over $\mfS_L$ (resp.\ over $\mfS$).
We put 
$\mfM_{L,s}=\mfS_{L,s}\otimes_{\mfS_L} \mfM$
(resp.\ $\mfS_L=\mfS_L\otimes_{\mfS} \mfM$)
and equip $\mfM_{L,s}$ (resp.\ $\mfM_L$) with a Frobenius by 
$\vphi=\vphi_{\mfS_{L,s}}\otimes \vphi_{\mfM}$
(resp.\ $\vphi=\vphi_{\mfS_L}\otimes \vphi_{\mfM}$).
Then it is not difficult to check  that 
$\mfM_{L,s}$ (resp.\ $\mfM_L$) is a free or torsion Kisin module of height $\le r$ 
over $\mfS_{L,s}$ (resp.\ over $\mfS_L$)
(here we recall that $E_s(u_s)=E(u^{p^s}_s)=E(u)$).
Hence we obtained natural functors 
$$
\mrm{Mod}^r_{/\mfS_L}\to \mrm{Mod}^r_{/\mfS_{L,s}} \quad  {\rm and}\quad 
\mrm{Mod}^r_{/\mfS_{L,\infty}}\to \mrm{Mod}^r_{/\mfS_{L,s,\infty}} 
$$
$$
{\rm (resp.}\quad 
\mrm{Mod}^r_{/\mfS}\to \mrm{Mod}^r_{/\mfS_L} \quad  {\rm and}\quad 
\mrm{Mod}^r_{/\mfS_{\infty}}\to \mrm{Mod}^r_{/\mfS_{L,\infty}}). 
$$
By definition,
we immediately see that we have 
$T_{\mfS_L}(\mfM)\simeq T_{\mfS_{L,s}}(\mfM_{L,s})$
(resp.\ $T_{\mfS}(\mfM)|_{G_{L_{\infty}}}\simeq T_{\mfS_L}(\mfM_L)$). 
In particular, it follows from Proposition \ref{Kisinfunctor} (1)
that the following holds:

\begin{proposition}
\label{basechange1:Kisin}
The functor
$\mrm{Mod}^r_{/\mfS_L}\to \mrm{Mod}^r_{/\mfS_{L,s}}$ is fully faithful.
\end{proposition}

\subsection{Base changes for $(\vphi,\hat{G})$-modules}
Let $\hat{\mfM}$ be a free or torsion $(\vphi,\hat{G}_L)$-module 
(resp.\ $(\vphi,\hat{G})$-module)
of height 
$\le r$ over $\mfS_L$ (resp.\ over $\mfS$).
The $G_{L,s}$ action on $\whR_L\otimes_{\vphi,\mfS_L} \mfM$ 
(resp.\ the $G_L$ action on $\whR\otimes_{\vphi,\mfS} \mfM$)
extends to 
$\whR_{L,s}\otimes_{\whR_L}(\whR_L\otimes_{\vphi,\mfS_L} \mfM)
\simeq \whR_{L,s}\otimes_{\vphi,\mfS_{L,s}} \mfM_{L,s}$ 
(resp.\ 
$\whR_L\otimes_{\whR}(\whR\otimes_{\vphi,\mfS} \mfM)
\simeq \whR_L\otimes_{\vphi,\mfS_L} \mfM_L$),
which factors through $\hat{G}_{L,s}$ (resp.\ $\hat{G}_L$).
Then it is not difficult to check that 
$\mfM_{L,s}$ (resp.\ $\mfM_L$) has a structure of a $(\vphi,\hat{G}_{L,s})$-module
(resp.\ $(\vphi,\hat{G}_L)$-module).
Hence we obtained natural functors 
$$
\mrm{Mod}^{r,\hat{G}_L}_{/\mfS_L}\to \mrm{Mod}^{r,\hat{G}_{L,s}}_{/\mfS_{L,s}} \quad  {\rm and}\quad 
\mrm{Mod}^{r,\hat{G}_L}_{/\mfS_{L,\infty}}\to \mrm{Mod}^{r,\hat{G}_{L,s}}_{/\mfS_{L,s,\infty}} 
$$
$$
{\rm (resp.}\quad 
\mrm{Mod}^{r,\hat{G}}_{/\mfS}\to \mrm{Mod}^{r,\hat{G}_L}_{/\mfS_L} \quad  {\rm and}\quad 
\mrm{Mod}^{r,\hat{G}}_{/\mfS_{\infty}}\to \mrm{Mod}^{r,\hat{G}_L}_{/\mfS_{L,\infty}}). 
$$
By definition,
we immediately see that we have 
$\hat{T}_L(\hat{\mfM})|_{G_{L,s}}\simeq \hat{T}_{L,s}(\hat{\mfM}_{L,s})$
(resp.\ $\hat{T}(\hat{\mfM})|_{G_L}\simeq \hat{T}_L(\hat{\mfM}_L)$).
Similar to Proposition \ref{basechange1:Kisin}, 
we can prove the following.
\begin{proposition}
The functor
$\mrm{Mod}^{r,\hat{G}_L}_{/\mfS_L}\to \mrm{Mod}^{r,\hat{G}_{L,s}}_{/\mfS_{L,s}} $
is fully faithful.
\end{proposition}

\noindent
The proposition immediately 
follows from the full faithfulness property of Theorem \ref{Thm1} (1)
and the lemma below.

\begin{lemma}
\label{totst}
Let $K'$ is a finite totally ramified extension of $K$.
Then the restriction functor from the category of 
semi-stable $\mbb{Q}_p$-representations of $G_K$
into the category of 
semi-stable $\mbb{Q}_p$-representations of $G_{K'}$ 
is fully faithful.
\end{lemma}
\begin{proof}
Let $V$ and $V'$ be semi-stable $\mbb{Q}_p$-representations of $G_K$
and let $f\colon V\to V'$ be a $G_{K'}$-equivariant homomorphism.
Considering the morphism of filtered $(\vphi, N)$-modules over $K'$ 
corresponding  to $f$, we can check without difficulty that 
$f$ is in fact a morphism of filtered $(\vphi, N)$-modules over $K$.
This is because $K'$ is totally ramified over $K_0$ as same as $K$.
This gives the desired result.
\end{proof}


\section{Variants of free $(\vphi,\hat{G})$-modules}

In this section, we define some variant notions of 
$(\vphi,\hat{G})$-modules.
We continue to use same notation as in the previous section. 
In particular, $p$ is odd.

\subsection{Definitions}
\label{vardef}

We start with some definitions which are our main concern in this and the next section.

\begin{definition}
\label{varLiumod}
We define the category $\Mod^{r,\hat{G}_{L,s}}_{/\mfS_L}$
(resp.\ $\wtMod^{r,\hat{G}_{L,s}}_{/\mfS_L}$) 
as follows.
An object 
is a triple $\hat{\mfM}=(\mfM, \vphi_{\mfM}, \hat{G}_{L,s})$ where
\begin{enumerate}
\item[(1)] $(\mfM, \vphi_{\mfM})$ is a free
           Kisin module of height $\le r$ over $\mfS_L$, 

\vspace{-2mm}
           
\item[(2)] $\hat{G}_{L,s}$ is an $\wh{\mcal{R}}_L$-semilinear
           $\hat{G}_{L,s}$-action on $\whR_L\otimes_{\vphi, \mfS_L} \mfM$
           (resp.\ an $\wh{\mcal{R}}_{L,s}$-semilinear
           $\hat{G}_{L,s}$-action on $\whR_{L,s}\otimes_{\vphi, \mfS_L} \mfM$)
           which induces a continuous $G_{L,s}$-action on 
           $W(\mrm{Fr}R)\otimes_{\vphi, \mfS_L} \mfM$
           for the weak topology,
\vspace{-2mm}           
           
\item[(3)] the $\hat{G}_{L,s}$-action commutes with 
           $\vphi_{\wh{\mcal{R}}_L}\otimes \vphi_{\mfM}$
           (resp.\ $\vphi_{\wh{\mcal{R}}_{L,s}}\otimes \vphi_{\mfM}$),

\vspace{-2mm}

\item[(4)] $\mfM\subset (\whR_L\otimes_{\vphi,\mfS_L} \mfM)^{H_L}$ 
           (resp.\ $\mfM\subset (\whR_{L,s}\otimes_{\vphi,\mfS_L} \mfM)^{H_L}$),

\vspace{-2mm}

\item[(5)] $\hat{G}_{L,s}$ acts on the 
           $W(k_L)$-module 
           $(\whR_L\otimes_{\vphi,\mfS_L} \mfM)/
            I_{+,L}(\whR_L\otimes_{\vphi,\mfS_L} \mfM)$
           (resp.\ $(\whR_{L,s}\otimes_{\vphi,\mfS_L} \mfM)/
            I_{+,L,s}(\whR_{L,s}\otimes_{\vphi,\mfS_L} \mfM)$) 
           trivially.           
\end{enumerate}
Morphisms are defined by the obvious way.
By replacing ``free'' of (1) with ``torsion''\footnote{As explained in the 
           footnote of Definition \ref{Liumod},
           we may ignore the continuity condition of (2) in the case where $\mfM$ is torsion.},
we define the category $\Mod^{r,\hat{G}_{L,s}}_{/\mfS_{L,\infty}}$ 
(resp.\ $\wtMod^{r,\hat{G}_{L,s}}_{/\mfS_{L,\infty}}$). 
\end{definition}

For any object $\hat{\mfM}$  of $\Mod^{r,\hat{G}_{L,s}}_{/\mfS_L}$ or 
$\Mod^{r,\hat{G}_{L,s}}_{/\mfS_{L,\infty}}$, 
we define a $\mbb{Z}_p[G_{L,s}]$-module $\hat{T}_{L,s}(\hat{\mfM})$
by
\begin{align*}
\hat{T}_{L,s}(\hat{\mfM})=
\left\{
\begin{array}{ll}
\mrm{Hom}_{\wh{\mcal{R}}_L,\vphi}
(\whR_L\otimes_{\vphi,\mfS_L} \mfM, W(R))
\hspace{21mm} {\rm if}\ \mfM\ {\rm is}\ {\rm free},  
\cr
\mrm{Hom}_{\whR_L,\vphi}
(\whR_L\otimes_{\vphi,\mfS_L} \mfM, \mbb{Q}_p/\mbb{Z}_p\otimes_{\mbb{Z}_p}W(R))
\hspace{3.5mm} {\rm if}\ \mfM\ {\rm is}\ {\rm torsion}.  
\end{array}
\right.
\end{align*}
Here,  $G_{L,s}$ 
acts on $\hat{T}_{L,s}(\hat{\mfM})$ by 
$(\sigma.f)(x)=\sigma(f(\sigma^{-1}(x)))$
for $\sigma\in G_{L,s},\ f\in \hat{T}_{L,s}(\hat{\mfM}),\ 
x\in \whR_L\otimes_{\vphi,\mfS_L}\mfM$. 
Similar to the above,
for any object $\hat{\mfM}$  of $\wtMod^{r,\hat{G}_{L,s}}_{/\mfS_L}$ or 
$\wtMod^{r,\hat{G}_{L,s}}_{/\mfS_{L,\infty}}$, 
we define a $\mbb{Z}_p[G_{L,s}]$-module $\hat{T}_{L,s}(\hat{\mfM})$
by
\begin{align*}
\hat{T}_{L,s}(\hat{\mfM})=
\left\{
\begin{array}{ll}
\mrm{Hom}_{\whR_{L,s},\vphi}
(\whR_{L,s}\otimes_{\vphi,\mfS_L} \mfM, W(R))
\hspace{21mm} {\rm if}\ \mfM\ {\rm is}\ {\rm free},  
\cr
\mrm{Hom}_{\wh{\mcal{R}}_{L,s},\vphi}
(\whR_{L,s}\otimes_{\vphi,\mfS_L} \mfM, \mbb{Q}_p/\mbb{Z}_p\otimes_{\mbb{Z}_p}W(R))
\hspace{3.5mm} {\rm if}\ \mfM\ {\rm is}\ {\rm torsion}.  
\end{array}
\right.
\end{align*}

On the other hand,
we obtain functors 
$
\Mod^{r,\hat{G}_L}_{/\mfS_L}\to \Mod^{r,\hat{G}_{L,s}}_{/\mfS_L}\to 
\wtMod^{r,\hat{G}_{L,s}}_{/\mfS_L}\to \Mod^{r,\hat{G}_{L,s}}_{/\mfS_{L,s}}
$
and 
$
\Mod^{r,\hat{G}_L}_{/\mfS_{L,\infty}}\to \Mod^{r,\hat{G}_{L,s}}_{/\mfS_{L,\infty}}\to 
\wtMod^{r,\hat{G}_{L,s}}_{/\mfS_{L,\infty}}\to \Mod^{r,\hat{G}_{L,s}}_{/\mfS_{L,s,\infty}}
$
by natural manners
and it is 
readily seen that
these functors are compatible with $\hat{T}_L$ and $\hat{T}_{L,s}$.
In particular,
the essential images of the functors $\hat{T}_{L,s}$ on $\Mod^{r,\hat{G}_{L,s}}_{/\mfS_L}$ 
and $\wtMod^{r,\hat{G}_{L,s}}_{/\mfS_L}$ has values in $\mrm{Rep}^{r,\mrm{st}}_{\mbb{Z}_p}(G_{L,s})$
since we have an equivalence of categories 
$\hat{T}_{L,s}\colon \Mod^{r,\hat{G}_{L,s}}_{/\mfS_{L,s}}
\overset{\sim}{\rightarrow} \mrm{Rep}^{r,\mrm{st}}_{\mbb{Z}_p}(G_{L,s})$
by Theorem \ref{Thm1}.

In the rest of this section,
we study free cases. 
We leave studies for torsion cases to the next section.\\

\noindent
{\bf Convention:}
For simplicity, 
if $L=K$, then we often omit the subscript ``$L$'' from various notations 
(e.g. $\Mod^{r,\hat{G}_{K,s}}_{/\mfS_K}=\Mod^{r,\hat{G}_s}_{/\mfS},  
\wtMod^{r,\hat{G}_{K,s}}_{/\mfS_K}=\wtMod^{r,\hat{G}_s}_{/\mfS}$).
Furthermore, if $s=0$, we often omit the subscript ``$s$'' from various notations 
(e.g. $\Mod^{r,\hat{G}_{L,0}}_{/\mfS_L}=\Mod^{r,\hat{G}_L}_{/\mfS_L},  
\wtMod^{r,\hat{G}_{L,0}}_{/\mfS_{L,0}}=\wtMod^{r,\hat{G}_L}_{/\mfS_L}$).\\


\subsection{The functors 
$
\Mod^{r,\hat{G}}_{/\mfS}\to \Mod^{r,\hat{G}_s}_{/\mfS}\to 
\wtMod^{r,\hat{G}_s}_{/\mfS}\to \Mod^{r,\hat{G}_s}_{/\mfS_s}
$}

Now we consider the functors 
$
\Mod^{r,\hat{G}}_{/\mfS}\to \Mod^{r,\hat{G}_s}_{/\mfS}\to 
\wtMod^{r,\hat{G}_s}_{/\mfS}\to \Mod^{r,\hat{G}_s}_{/\mfS_s}.
$
At first, by Proposition \ref{basechange1:Kisin},
we see that the functor $\wtMod^{r,\hat{G}_s}_{/\mfS}\to \Mod^{r,\hat{G}_s}_{/\mfS_s}$
is fully faithful. 
It follows from this fact and Theorem \ref{Thm1} (1) that the functor
$\hat{T}_s\colon \wtMod^{r,\hat{G}_s}_{/\mfS}
\to \mrm{Rep}^{r,\mrm{st}}_{\mbb{Z}_p}(G_s)$ is fully faithful.
It is clear that the functor $\Mod^{r,\hat{G}_s}_{/\mfS}\to 
\wtMod^{r,\hat{G}_s}_{/\mfS}$ is fully faithful
and thus so is $\hat{T}_s\colon \Mod^{r,\hat{G}_s}_{/\mfS}
\to \mrm{Rep}^{r,\mrm{st}}_{\mbb{Z}_p}(G_s)$.
Combining this with Theorem \ref{Thm1} (1) and Lemma \ref{totst},
we obtain that the functor 
$\Mod^{r,\hat{G}}_{/\mfS}\to \Mod^{r,\hat{G}_s}_{/\mfS}$
is also fully faithful.
Furthermore, we prove the following.

\begin{proposition}
\label{equal}
The functor $\Mod^{r,\hat{G}_s}_{/\mfS}\to 
\wtMod^{r,\hat{G}_s}_{/\mfS}$
is an equivalence of categories.
\end{proposition}

\noindent
Summary, we obtained the following commutative diagram.
\begin{center}
$\displaystyle \xymatrix{
\Mod^{r,\hat{G}}_{/\mfS}
\ar@{^{(}->}[r] \ar[d]_{\wr} \ar^{\hat{T}}[d]
& 
\Mod^{r,\hat{G}_s}_{/\mfS}
\ar^{\sim}[r] \ar@{^{(}->}[rrd] \ar^{\hat{T}_s}[rrd]
&
\wtMod^{r,\hat{G}_s}_{/\mfS} 
\ar@{^{(}->}[r] \ar@{^{(}->}[rd] \ar^{\hat{T}_s}[rd]
&
\Mod^{r,\hat{G}_s}_{/\mfS_s} 
\ar[d]_{\wr} \ar^{\hat{T}_s}[d]\\
\mrm{Rep}^{r,\mrm{st}}_{\mbb{Z}_p}(G_K) 
\ar@{^{(}->}[rrr]  \ar^{\mrm{restriction}}[rrr]   
& 
&
& 
\mrm{Rep}^{r,\mrm{st}}_{\mbb{Z}_p}(G_s).
}$
\end{center}

\begin{remark}
The functor 
$\wtMod^{r,\hat{G}_s}_{/\mfS}\hookrightarrow \Mod^{r,\hat{G}_s}_{/\mfS_s}$
may not be possibly essentially surjective.
In fact, under some conditions, there exists a 
representation of $G_K$ which is 
crystalline over $K_s$
but not of finite height.
For more precise information, see \cite[Example 4.2.3]{Li2}.
\end{remark}

Before a proof of Proposition \ref{equal},
we give an explicit formula such as (\ref{explicit}) for    
an object of $\wtMod^{r,\hat{G}_s}_{/\mfS}$.
The argument below follows the method of \cite{Li2}.
Let $\hat{\mfM}$
be an object of $\wtMod^{r,\hat{G}_s}_{/\mfS}$.
Let $\hat{\mfM}_s$ be the image of $\hat{\mfM}$ for the functor 
$\wtMod^{r,\hat{G}_s}_{/\mfS}\to \Mod^{r,\hat{G}_s}_{/\mfS_s}$.
Put $\mcal{D}=S_{K_0}\otimes_{\vphi,\mfS} \mfM$ and also put  
$\mcal{D}_s=S_{K_0,s}\otimes_{\vphi,\mfS_s} \mfM_s
=S_{K_0,s}\otimes_{S_{K_0}} \mcal{D}$.
Then $\mcal{D}_s$ has a structure of a Breuil module and also 
$D=\mcal{D}_s/I_+S_{K_0,s}\mcal{D}_s$ has a structure 
of a filtered $(\vphi,N)$-module
corresponding to
$\mbb{Q}_p\otimes_{\mbb{Z}_p} \hat{T}_s(\hat{\mfM}_s)$ (see subsection \ref{relations}),
which is isomorphic to $\mcal{D}/I_+S_{K_0}\mcal{D}$ 
as a $\vphi$-module over $K_0$.
By \cite[Lemma 7.3.1]{Li1},
we have a unique $\vphi$-compatible section $D\hookrightarrow \mcal{D}$
and we regard $D$ as a submodule of $\mcal{D}\subset \mcal{D}_s$ by this section.
Then we have $\mcal{D}=S_{K_0} \otimes_{K_0} D$ and 
$\mcal{D}_s=S_{K_0,s} \otimes_{K_0} D$.
By the explicit formula (\ref{explicit}) for $\hat{\mfM}_s$,
we know that
$$
\hat{G}_s(D)\subset (K_0[\![t]\!]\cap \mcal{R}_{K_0,s})\otimes_{K_0} D.
$$
Hence, taking any $K_0$-basis $e_1,\dots ,e_d$ of $D$,
there exist $A_s(t)\in M_{d\times d}(K_0[\![t]\!])$
such that 
$\tau^{p^s}(e_1,\cdots ,e_d)=(e_1,\dots ,e_d)A_s(t)$.
Since $A_s(0)=\mrm{I}_d$, we see that  
$\mrm{log}(A_s(t))\in M_{d\times d}(K_0[\![t]\!])$
is well-defined.
On the other hand,
choose $g_0\in G_s$ such that $\chi_p(g_0)\not=1$, where $\chi_p$ is the 
$p$-adic cyclotomic character. 
Since $g_0\tau^{p^s}=(\tau^{p^s})^{\chi_p(g_0)}g_0$,
we have $A_s(\chi_p(g_0)t)=A_s(t)^{\chi_p(g_0)}$ and 
thus we also have $\mrm{log}(A_s(\chi_p(g_0)t))=\chi_p(g_0)\mrm{log}(A_s(t))$.
Since $\mrm{log}(A_s(0))=\mrm{log}(I_d)=0$, 
we can write $\mrm{log}(A_s(t))$ as $tB(t)$ for some $B(t)\in M_{d\times d}(K_0[\![t]\!])$.
Then we have $\chi_p(g_0)tB(\chi_p(g_0)t)=\chi_p(g_0)tB(t)$, that is, $B(\chi_p(g_0)t)=B(t)$.
Hence the assumption $\chi_p(g_0)\not= 1$ implies that $B(t)$ is a constant.
Putting $N_s=B(t)\in M_{d\times d}(K_0)$, we obtain 
$$
\tau^{p^s}(e_1,\cdots ,e_d)=(e_1,\cdots ,e_d)(\sum_{i\ge 0}N^i_s\gamma_i(t)).
$$
Now we define $N_D\colon D\to D$ by 
$N(e_1,\cdots ,e_d)=(e_1,\cdots ,e_d)p^{-s}N_s$
and also define 
$N_{\mcal{D}}:=N_{S_{K_0}}\otimes \mrm{Id}_D+\mrm{Id}_{S_{K_0}}\otimes N_D$.
(Note that we have $N_D\vphi_D=p\vphi_D N_D$ and thus $N_D$ is nilpotent.)
It is a routine work to check the following:
\begin{equation}
\label{explicit''}
g(a\otimes x)=\sum^{\infty}_{i=0}g(a)\gamma_i
(-\mrm{log}([\underline{\e}(g)]))\otimes N^i_D(x)\quad
{\rm for}\ g\in G_s, a\in B^+_{\mrm{cris}}, x\in D.
\end{equation}
Since we have 
\begin{equation}
\label{easyeq}
g(f)=\sum_{i\ge 0} \gamma_i(-\mrm{log}([\underline{\e}(g)]))N^i_{S_{K_0}}(f)
\end{equation}
for any $g\in G_K$ and $f\in S_{K_0}$,
we obtain the following explicit formula:
\begin{equation}
\label{explicit'}
g(a\otimes x)=\sum^{\infty}_{i=0}g(a)\gamma_i
(-\mrm{log}([\underline{\e}(g)]))\otimes N^i_{\mcal{D}}(x)\quad
{\rm for}\ g\in G_s, a\in B^+_{\mrm{cris}}, x\in \mcal{D}.
\end{equation}
In particular, as in subsection \ref{relations}, 
we can show that 
\begin{equation}
\label{eq2}
\mrm{log}(\tau^{p^s})(x)=p^st\otimes N_{\mcal{D}}(x)
\end{equation} 
for any $x\in \mcal{D}$.

\begin{proof}[Proof of Proposition \ref{equal}]
We continue to use the above notation.
It suffices to prove that the $G_s$-action on $\whR_s\otimes_{\vphi,\mfS} \mfM$
preserves $\whR\otimes_{\vphi,\mfS}\mfM$.
Take any $g\in G_s$.
We know that $g(\mfM)\subset \whR_s\otimes_{\vphi,\mfS} \mfM\subset W(R)\otimes_{\vphi,\mfS} \mfM$.
Hence it is enough to  prove that 
$g(\mcal{D})\in \mcal{R}_{K_0}\otimes_{\vphi,\mfS} \mfM$.
Let $s\in S^{\mrm{int}}_{K_0}$ and $y\in D$ and 
put $x=s\otimes y\in S^{\mrm{int}}_{K_0}\otimes_{W(k)} D=\mcal{D}$.
By (\ref{explicit''}) or (\ref{explicit'}),
we have 
$$
g(x)=\sum_{i\ge 0}\sum_{0\le j\le i} \gamma_i(-\mrm{log}([\underline{\e}(g)]))
\binom{i}{j}N^{i-j}_{S_{K_0}}(s)\otimes N^j_D(y)
$$
On the other hand, we know that $N_D$ is nilpotent, that is,
there exists $j_0>0$ such that $N^{j_0}_D=0$.
Then we obtain
$$
g(x)=\sum_{0\le j\le j_0} \sum^{\infty}_{i=j} \gamma_i(-\mrm{log}([\underline{\e}(g)]))
     \binom{i}{j}N^{i-j}_{S_{K_0}}(s)\otimes N^j_D(y).
$$
Therefore, it suffices to show that 
$\sum^{\infty}_{i=j} \gamma_i(-\mrm{log}([\underline{\e}(g)]))
     \binom{i}{j}N^{i-j}_{S_{K_0}}(s)$
     is contained in $\mcal{R}_{K_0}$
for each $0\le j\le j_0$.
Taking $\alpha(g)\in \mbb{Z}_p$ such that 
$\mrm{log}([\underline{\e}(g)])=-\alpha(g)t$,
we have
$$
     \sum^{\infty}_{i=j} \gamma_i(-\mrm{log}([\underline{\e}(g)]))
     \binom{i}{j}N^{i-j}_{S_{K_0}}(s)
     =
     \sum^{\infty}_{i=j} (\alpha(g))^i\frac{\tilde{q}(i)!p^{\tilde{q}(i)}}{i!}
     \binom{i}{j}N^{i-j}_{S_{K_0}}(s)t^{\{i\}}.
$$
Since $\frac{\tilde{q}(i)p^{\tilde{q}(i)}}{i!}\to 0$ ($p$-adically) 
as $i\to \infty$, we finish a proof.
\end{proof}


\subsection{Relations with crystalline representations}

We know that $\mbb{Q}_p\otimes_{\mbb{Z}_p} \hat{T}_s(\hat{\mfM})$
is semi-stable over $K_s$ for any object $\hat{\mfM}$ of 
$\Mod^{r,\hat{G}_s}_{/\mfS}$ or $\wtMod^{r,\hat{G}_s}_{/\mfS}$.
This subsection is devoted to prove a criterion, for $\hat{\mfM}$,
that describes when $\mbb{Q}_p\otimes_{\mbb{Z}_p} \hat{T}_s(\hat{\mfM})$ 
becomes crystalline. 

Following \cite[Section 5]{Fo2}
we set $I^{[m]}B^+_{\mrm{cris}}:=\{x\in B^+_{\mrm{cris}} \mid \vphi^n(x)
\in \mrm{Fil}^mB^+_{\mrm{cris}}\ {\rm for\ all}\ n\ge0 \}$.
For any subring $A\subset B^+_{\mrm{cris}}$,
we put $I^{[m]}A=A\cap I^{[m]}B^+_{\mrm{cris}}$.
Furthermore, we also put $I^{[m+]}A=I^{[m]}A.I_+A$
(here, $I_+A$ is defined in Subsection \ref{Liumodule:section}).
By \cite[Proposition 5.1.3]{Fo2} and the proof of \cite[Lemma 3.2.2]{Li2},
we know that $I^{[m]}W(R)$ is a principal ideal which is generated by $\vphi(\mft)^m$.

Now we recall Theorem \ref{Thm1} (2):
if $\mfM$ is an object of $\Mod^{r,\hat{G}_s}_{/\mfS_s}$,
then $\mbb{Q}_p\otimes_{\mbb{Z}_p} \hat{T}_s(\hat{\mfM})$
is crystalline if and only if 
$\tau^{p^s}(x)-x\in u_s^p(I^{[1]}W(R)\otimes_{\vphi,\mfS_s} \mfM)$ for any $x\in \mfM$.
However,
if such $\mfM$ descends to a Kisin module over $\mfS$,
then we can show the following.

\begin{theorem}
\label{cris}
Let $\hat{\mfM}$ be an object of $\Mod^{r,\hat{G}_s}_{/\mfS}$ or $\wtMod^{r,\hat{G}_s}_{/\mfS}$.
Then the following is equivalent:
\vspace{-2mm}
\begin{enumerate}
\item[$(1)$] $\mbb{Q}_p\otimes_{\mbb{Z}_p} \hat{T}_s(\hat{\mfM})$ is crystalline, 

\vspace{-2mm}
           
\item[$(2)$] $\tau^{p^s}(x)-x\in u^p(I^{[1]}W(R)\otimes_{\vphi,\mfS} \mfM)$ for any $x\in \mfM$,

\vspace{-2mm}           
           
\item[$(3)$] $\tau^{p^s}(x)-x\in I^{[1+]}W(R)\otimes_{\vphi,\mfS} \mfM$ for any $x\in \mfM$.
\end{enumerate}
\end{theorem}

\begin{proof}
(1) $\Rightarrow$ (2):
The proof here mainly follows that of \cite[Proposition 4.7]{GLS}.
We may suppose $\hat{\mfM}$ is an object of $\wtMod^{r,\hat{G}_s}_{/\mfS}$.
Put $\mcal{D}=S_{K_0}\otimes_{\vphi,\mfS} \mfM$ 
and $D=\mcal{D}/I_+S_{K_0} \mcal{D}$ as in the previous subsection.
We fix a $\vphi(\mfS)$-basis $(\hat{e}_1,\dots ,\hat{e}_d)$ of 
$\mfM\subset \mcal{D}$ and denote by $(e_1,\dots ,e_d)$ the image of $(\hat{e}_1,\dots ,\hat{e}_d)$
for the projection $\mcal{D}\to D$. Then $(e_1,\dots ,e_d)$ is a $K_0$-basis of $D$.
As described before the proof of Proposition \ref{equal},
we regard $D$ as a $\vphi$-stable submodule of $\mcal{D}$, and
we have $N_D\colon D\to D$ and $N_{\mcal{D}}\colon D_{\mcal{D}}\to D_{\mcal{D}}$.

Now we consider a matrix $X\in GL_{d\times d}(S_{K_0})$ 
such that $(\hat{e}_1,\dots ,\hat{e}_d)=(e_1,\dots ,e_d)X$. 
We define $\tilde{S}=W(k)[\![u^p, u^{ep}/p]\!]$ as in Section 4 of \cite{GLS},
which is a sub $W(k)$-algebra of $S^{\mrm{int}}_{K_0}$ with the property 
$N_{S_{K_0}}(\tilde{S})\subset u^p\tilde{S}$.
By an easy computation we have 
$U=X^{-1}BX+X^{-1}N_{S_{K_0}}(X)$. Here, 
$B\in M_{d\times d}(K_0)$ and $U\in M_{d\times d}(S_{K_0})$ are defined by 
$N_D(e_1,\dots ,e_d)=(e_1,\dots ,e_d)B$ and 
$N_\mcal{D}(\hat{e}_1,\dots ,\hat{e}_d)=(\hat{e}_1,\dots ,\hat{e}_d)U$. 
By the same proof as in the former half part of the proof of \cite[Proposition 4.7]{GLS},
we obtain $X,X^{-1}\in M_{d\times d}(\tilde{S}[1/p])$.
On the other hand, let 
$\hat{\mfM}_s$ be the image of $\hat{\mfM}$ for the 
functor $\wtMod^{r,\hat{G}_s}_{/\mfS}\to \Mod^{r,\hat{G}_s}_{\mfS_s}$.
Now we recall that 
$\mcal{D}_s=S_{K_0,s}\otimes_{\vphi,\mfS_s} \mfM_s$ 
has a structure of the Breuil module corresponding 
to $\mbb{Q}_p\otimes_{\mbb{Z}_p} \hat{T}_s(\hat{\mfM}_s)$
Denote by $N_{\mcal{D}_s}$ its monodromy operator.
By the formula (\ref{eq1}) for $\hat{\mfM}_s$   
and the formula (\ref{eq2}) for $\hat{\mfM}$,
we see that $p^sN_{\mcal{D}}=N_{\mcal{D}_s}$ on $\mcal{D}$.
Therefore, $\mbb{Q}_p\otimes_{\mbb{Z}_p} \hat{T}_s(\hat{\mfM})$
is crystalline if and only if $N_{\mcal{D}_s}$ mod $I_+S_{K_0,s}\mcal{D}_s$ is zero,
which is equivalent to say that $N_D=(N_{\mcal{D}}$ mod $I_+S_{K_0}\mcal{D})$ is zero, that is, $B=0$.
Therefore, the latter half part of the proof 
\cite[Proposition]{GLS} gives the assertion (2). 

\noindent 
(2) $\Rightarrow$ (3): This is clear.

\noindent
(3) $\Rightarrow$ (1):
Suppose that (3) holds.
We denote by $\hat{\mfM}_s$  the image of $\hat{\mfM}$ for the functor 
$\wtMod^{r,\hat{G}_s}_{/\mfS}\to \Mod^{r,\hat{G}_s}_{\mfS_s}$ as above.
We claim that,
for any $x\in \mfM_s$, we have 
$\tau^{p^s}(x)-x\in I^{[1+]}W(R)\otimes_{\vphi_s} \mfM_s$.
Let $x=a\otimes y\in \mfM_s=\mfS_s\otimes_{\mfS} \mfM$ where $a\in \mfS_s$ and $y\in \mfM$.
Then 
$$
\tau^{p^s}(x)-x=\tau^{p^s}(\vphi(a))(\tau^{p^s}(y)-y)+(\tau^{p^s}(\vphi(a))-\vphi(a))y
$$
and thus it suffices to show $\tau^{p^s}(\vphi(a))-\vphi(a)\in I^{[1+]}W(R)$.
This follows from the lemma below and thus we obtained the claim.
By the claim and Theorem \ref{Thm1} (2),
we know that $\mbb{Q}_p\otimes_{\mbb{Z}_p} \hat{T}_s(\hat{\mfM}_s)\simeq 
\mbb{Q}_p\otimes_{\mbb{Z}_p} \hat{T}_s(\hat{\mfM})$ is crystalline.
\end{proof}

\begin{lemma}
\label{cryslem}
$(1)$ We have $I^{[1]}W(R)\cap u^{\ell}B^+_{\mrm{cris}}
=u^{\ell}I^{[1]}W(R)$ for $\ell\ge 0$.

\noindent
$(2)$
We have $g(a)-a\in uI^{[1]}W(R)$
for $g\in G$ and $a\in \mfS$.
\end{lemma}

\begin{proof}
This is due to \cite[the proof of Proposition 7]{GLS}
but we write a proof here.

\noindent
(1) 
Take $x=u^{\ell}y\in I^{[1]}W(R)$ with $y\in B^+_{\mrm{cris}}$.
By Lemma 3.2.2 of \cite{Li3} we have $y\in W(R)$.
Now we remark that $uz\in \mrm{Fil}^nW(R)$ with $z\in W(R)$ implies 
$z\in \mrm{Fil}^nW(R)$
since $u$ is a unit of $B^+_{\mrm{dR}}$.
Hence $u^{\ell}y\in I^{[1]}W(R)$ implies $y\in I^{[1]}W(R)$.

\noindent
(2) 
By the relation (\ref{easyeq}),
we see that $g(a)-a\in I^{[1]}W(R)$.
On the other hand, if $i>0$, we can write $N^i_{S_{K_0}}(a)=ub_i$ for some $b_i\in \mfS$.
Thus by the relation (\ref{easyeq}) again we obtain 
$g(a)-a\in uB^+_{\mrm{cris}}$.
Then the result follows from (1).
\end{proof}


\section{Variants of torsion $(\vphi,\hat{G})$-modules}

In this section,
we mainly study full subcategories of 
$\wtMod^{r,\hat{G}_s}_{/\mfS_{\infty}}$ defined below and 
also study representations associated with them.
As a consequence,
we prove theorems in Introduction.
We use same notation as in Section 2 and 3.
In particular, $p$ is odd.
In below, 
let $v_R$ be the valuation of $R$ normalized such that $v_R(\underline{\pi})=1/e$
and, for any real number $x\ge 0$, we denote by $\mfm^{\ge x}_R$ 
the ideal of $R$ consisting of elements $a$ with $v_R(a)\ge x$.

Let $J$ be an ideal of $W(R)$ which satisfies the following conditions:
\begin{itemize}
\item $J\not\subset pW(R)$,
\item $J$ is a principal ideal, 
\item $J$ is $\vphi$-stable and $G_s$-stable in $W(R)$. 
\end{itemize}
By the above first and second assumptions for $J$, 
the image of $J$ under the projection $W(R)\twoheadrightarrow R$ is of the form 
$\mfm^{\ge c_J}_{R}$ for some real number $c_J\ge 0$.

\begin{definition}
We denote by $\wtMod^{r,\hat{G}_s,J}_{/\mfS_{\infty}}$
the full subcategory 
of $\wtMod^{r,\hat{G}_s}_{/\mfS_{\infty}}$
consisting of objects $\hat{\mfM}$ which satisfy the following condition:
$$
\tau^{p^s}(x)-x\in JW(R)\otimes_{\vphi, \mfS} \mfM\quad {\rm for\ any}\ x\in \mfM.
$$ 
Also, we denote by $\wt{\mrm{Rep}}^{r,\hat{G}_s, J}_{\mrm{tor}}(G_s)$
the essential image of the functor 
$\hat{T}_s\colon \wtMod^{r,\hat{G}_s}_{/\mfS_{\infty}}\to \mrm{Rep}_{\mrm{tor}}(G_s)$
restricted to $\wtMod^{r,\hat{G}_s,J}_{/\mfS_{\infty}}$.
\end{definition}

\noindent
By definition,
we have 
$\wtMod^{r,\hat{G}_s,J}_{/\mfS_{\infty}}\subset \wtMod^{r,\hat{G}_s,J'}_{/\mfS_{\infty}}$
and 
$\wt{\mrm{Rep}}^{r,\hat{G}_s,J}_{\mrm{tor}}(G_s)
\subset \wt{\mrm{Rep}}^{r,\hat{G}_s,J'}_{\mrm{tor}}(G_s)$
for $J\subset J'$.


\subsection{Full faithfullness for $\wtMod^{r,\hat{G}_s,J}_{/\mfS_{\infty}}$}

For the beginning of a study of $\wtMod^{r,\hat{G}_s,J}_{/\mfS_{\infty}}$,
we prove the following full faithfullness result.
\begin{proposition}
\label{FFTHMMOD}
Let $r$ and $r'$ be non-negative integers with $c_J> pr/(p-1)$.
Let $\hat{\mfM}$ and $\hat{\mfN}$ be objects of 
$\wtMod^{r,\hat{G}_s,J}_{/\mfS_{\infty}}$ and 
$\wtMod^{r',\hat{G}_s,J}_{/\mfS_{\infty}}$,
respectively.
Then we have $\mrm{Hom}(\hat{\mfM}, \hat{\mfN})=\mrm{Hom}(\mfM, \mfN)$.
$($Here, two ``$\mrm{Hom}$''s are defined by obvious manners.$)$

In particular, if $c_J> pr/(p-1)$, then the forgetful functor 
$\wtMod^{r,\hat{G}_s,J}_{/\mfS_{\infty}}\to \Mod^r_{/\mfS_{\infty}}$
is fully faithful.
\end{proposition}

\begin{proof}
A very similar proof of \cite[Lemma 7]{Oz2} proceeds, and hence we
only give a sketch here.
Let $\hat{\mfM}$ and $\hat{\mfN}$ be objects of 
$\wtMod^{r,\hat{G}_s,J}_{/\mfS_{\infty}}$
and 
$\wtMod^{r',\hat{G}_s,J}_{/\mfS_{\infty}}$,
respectively.
Let $f\colon \mfM\to \mfN$ be a morphism of Kisin modules over $\mfS$.
Put $\hat{f}=W(R)\otimes f\colon W(R)\otimes_{\vphi,\mfS}\mfM\to W(R)\otimes_{\vphi,\mfS}\mfM$.
Choose any lift of $\tau\in \hat{G}$ to $G_K$;
we denote it also by $\tau$.
Since the $\hat{G}_s$-action for $\hat{\mfM}$ is continuous, 
it suffices to prove that 
$\Delta(1\otimes x)=0$  for any $x\in \mfM$
where $\Delta:=\tau^{p^s}\circ \hat{f}-\hat{f}\circ \tau^{p^s}$.
We use induction on $n$ such that $p^n\mfN=0$.

Suppose $n=1$.
Since $\Delta=(\tau^{p^s}-1)\circ \hat{f}-\hat{f}\circ (\tau^{p^s}-1)$,
we obtain the following:
\begin{center}
$(0)$:\quad  For any $x\in \mfM$, 
$\Delta(1\otimes x)\in 
\mfm^{\ge c(0)}_R(R\otimes_{\vphi,\mfS} \mfN)$
\end{center}
where $c(0)=c_J$.
Since $\mfM$ is of height $\le r$,
we further obtain the following for any $i\ge 1$ inductively:
\begin{center}
$(i)$:\quad  For any $x\in \mfM$, 
$\Delta(1\otimes x)\in 
\mfm^{\ge c(i)}_R(R\otimes_{\vphi,\mfS} \mfN)$
\end{center}
where $c(i)=pc(i-1)-pr
=(c_J-pr/(p-1))p^i+pr/(p-1)$.
The condition $c_J>pr/(p-1)$ implies that $c(i)\to \infty$ as $i\to \infty$
and thus $\Delta(1\otimes x)=0$. 

Suppose $n>1$.
Consider the exact sequence 
$0\to \mrm{Ker}(p)\to \mfN\overset{p}{\to} p\mfN\to 0$ of $\vphi$-modules over $\mfS$.
It is not difficult to check that
$\mfN':=\mrm{Ker}(p)$ and $\mfN'':=p\mfN$ are torsion 
Kisin modules of height $\le r'$ over $\mfS$
(cf.\ \cite[Lemma 2.3.1]{Li1}).
Moreover, we can check that $\mfN'$ and $\mfN''$ have natural structures
of objects of $\wtMod^{r',\hat{G}_s}_{/\mfS_{\infty}}$
(which are denoted by  $\hat{\mfN}'$ and $\hat{\mfN}''$, respectively)
such that 
the sequence $0\to \mfN'\to \mfN\overset{p}{\to} \mfN''\to 0$ induces an exact sequence
$0\to \hat{\mfN}'\to \hat{\mfN}\to \hat{\mfN}''\to 0$.
By the lemma below,
we know that  $\hat{\mfN}'$ and $\hat{\mfN}''$ are in fact contained in
$\wtMod^{r',\hat{G}_s,J}_{/\mfS_{\infty}}$.
By the induction hypothesis, we see that 
$\Delta(1\otimes x)$ has values in 
$(W(R)\otimes_{\vphi, \mfS} \mfN')\cap 
(JW(R)\otimes_{\vphi, \mfS} \mfN)$.
By Lemma 6 of \cite{Oz2} and the assumption that $J\not\subset pW(R)$ is principal,
we obtain that $\Delta(1\otimes x)\in JW(R)\otimes_{\vphi, \mfS} \mfN'$.
Since $p\mfN'=0$,
an analogous argument in the case $n=1$ proceeds
and we have $\Delta(1\otimes x)=0$ as desired.
\end{proof}

\begin{lemma}
\label{speciallemma}
Let $0\to \hat{\mfM}'\to \hat{\mfM}\to \hat{\mfM}''\to 0$
be an exact sequence in 
$\wtMod^{r,\hat{G}_s}_{/\mfS_{\infty}}$.
Suppose that $\hat{\mfM}$ is an object of 
$\wtMod^{r,\hat{G}_s,J}_{/\mfS_{\infty}}$.
Then $\hat{\mfM}'$ and $\hat{\mfM}''$ are also objects of 
$\wtMod^{r,\hat{G}_s,J}_{/\mfS_{\infty}}$.
\end{lemma}

\begin{proof}
The fact $\hat{\mfM}''\in \wtMod^{r,\hat{G}_s,J}_{/\mfS_{\infty}}$
is clear.
Take any $x\in \mfM'$.
Then we have 
$\tau^{p^s}(x)-x\in (JW(R)\otimes_{\vphi,\mfS}\mfM)\cap
(W(R)\otimes_{\vphi,\mfS}\mfM')$.
Since $J$ is a principal ideal which is not contained in $pW(R)$,
we obtain 
$\tau^{p^s}(x)-x\in JW(R)\otimes_{\vphi,\mfS}\mfM'$
by Lemma 6 of \cite{Oz2}.
This implies $\hat{\mfM}'\in \wtMod^{r,\hat{G}_s,J}_{/\mfS_{\infty}}$.
\end{proof}


\subsection{The category $\wt{\mrm{Rep}}^{r,\hat{G}_s,J}_{\mrm{tor}}(G_s)$}

In this subsection, 
we study some categorical properties of 
$\wt{\mrm{Rep}}^{r,\hat{G}_s,J}_{\mrm{tor}}(G_s)$.

Let $\hat{\mfM}$ be an object of $\wtMod^{r,\hat{G}_s}_{/\mfS_{\infty}}$.
Following Section 3.2 of \cite{Li2} (note that arguments in 
\cite{Li2} is the ``free case''),
we construct a map $\hat{\iota}_s$ which connects 
$\hat{\mfM}$ and $\hat{T}_s(\hat{\mfM})$ as follows.
Observe that there exists a natural isomorphism of 
$\mbb{Z}_p[G_s]$-modules
\[
\hat{T}_s(\hat{\mfM})
\simeq
\mrm{Hom}_{W(R),\vphi}(W(R)\otimes_{\vphi,\mfS}\mfM, \mbb{Q}_p/\mbb{Z}_p\otimes_{\mbb{Z}_p} W(R))
\]
where $G_s$ acts on 
$\mrm{Hom}_{W(R),\vphi}(W(R)\otimes_{\vphi,\mfS}\mfM,\mbb{Q}_p/\mbb{Z}_p\otimes_{\mbb{Z}_p} W(R))$
by $(\sigma.f)(x)=\sigma(f(\sigma^{-1}(x)))$ for 
$\sigma\in G_s,
f\in \mrm{Hom}_{W(R),\vphi}(W(R)\otimes_{\vphi,\mfS}\mfM,\mbb{Q}_p/\mbb{Z}_p\otimes_{\mbb{Z}_p} W(R)),
x\in W(R)\otimes_{\vphi,\mfS}\mfM=W(R)\otimes_{\whR_s} (\whR_s\otimes_{\vphi,\mfS}\mfM)$.
Thus we can define a morphism 
$\hat{\iota}'_s\colon W(R)\otimes_{\vphi, \mfS}\mfM \to 
\mrm{Hom}_{\mbb{Z}_p}(\hat{T}_s(\hat{\mfM}),\mbb{Q}_p/\mbb{Z}_p\otimes_{\mbb{Z}_p} W(R))$
by
$$
x\mapsto (f\mapsto f(x)),\quad 
x\in W(R)\otimes_{\vphi,\mfS}\mfM, f\in \hat{T}_s(\hat{\mfM}). 
$$
Since $\hat{T}_s(\hat{\mfM})\simeq \oplus_{i\in I}\mbb{Z}_p/p^{n_i}\mbb{Z}_p$ 
as $\mbb{Z}_p$-modules,
we have a natural isomorphism 
$\mrm{Hom}_{\mbb{Z}_p}(\hat{T}_s(\hat{\mfM}),\mbb{Q}_p/\mbb{Z}_p\otimes_{\mbb{Z}_p} W(R))\simeq 
W(R)\otimes_{\mbb{Z}_p}\hat{T}_s^{\vee}(\hat{\mfM})$ 
where $\hat{T}_s^{\vee}(\hat{\mfM})=\mrm{Hom}_{\mbb{Z}_p}(\hat{T}_s(\hat{\mfM}),\mbb{Q}_p/\mbb{Z}_p)$
is the dual representation of $\hat{T}_s(\hat{\mfM})$.
Composing this isomorphism with $\hat{\iota}'_s$, we obtain the desired map 
$$
\hat{\iota}_s\colon W(R)\otimes_{\vphi,\mfS}\mfM\to 
W(R)\otimes_{\mbb{Z}_p}\hat{T}_s^{\vee}(\hat{\mfM}).
$$
It follows from a direct calculation  that
$\hat{\iota}_s$ is $\vphi$-equivariant and $G_s$-equivariant.
If we denote by $\hat{\mfM}_s$ the image of $\hat{\mfM}$ for the functor 
$\wtMod^{r,\hat{G}_s}_{/\mfS_{\infty}}\to 
\Mod^{r,\hat{G}_s}_{/\mfS_{s,\infty}}$ (cf.\ Section \ref{vardef}),
then the above $\hat{\iota}_s$ is isomorphic to 
``$\hat{\iota}$ for $\hat{\mfM}_s$ in Section 4.1 of \cite{Oz1}''.
Hence Lemma 4.2 (4) in {\it loc}.\ {\it cit}. implies that 
$$
W(\mrm{Fr}\ R)\otimes \hat{\iota}_s\colon 
W(\mrm{Fr}\ R)\otimes_{W(R)}(W(R)\otimes_{\vphi,\mfS}\mfM)
\to 
W(\mrm{Fr}\ R)\otimes_{W(R)}(W(R)\otimes_{\mbb{Z}_p}\hat{T}_s^{\vee}(\hat{\mfM}))
$$
is  bijective.

\begin{proposition}
Let $(R) \colon 0\to T'\to T\to T''\to 0$ be an exact sequence in $\mrm{Rep}_{\mrm{tor}}(G_s)$.
Assume that there exists $\hat{\mfM}\in \wtMod^{r,\hat{G}_s,J}_{/\mfS_{\infty}}$
such that $\hat{T}_s(\hat{\mfM})\simeq T$.
Then there exists an exact sequence 
$(M) \colon 0\to \hat{\mfM}''\to \hat{\mfM}\to \hat{\mfM}'\to 0$
in $\wtMod^{r,\hat{G}_s,J}_{/\mfS_{\infty}}$
such that $\hat{T}_s((M))\simeq (R)$. 
\end{proposition}

\begin{proof}
The same proof as \cite[Theorem 4.5]{Oz1},
except using not $\hat{\iota}$ in the proof of {\it loc.}\ {\it cit.}
but $\hat{\iota}_s$ as above,
gives an exact sequence $(M) \colon 0\to \hat{\mfM}''\to \hat{\mfM}\to \hat{\mfM}'\to 0$
in $\wtMod^{r,\hat{G}_s}_{/\mfS_{\infty}}$
such that $\hat{T}_s((M))\simeq (R)$.
Therefore, Lemma \ref{speciallemma},
gives the desired result.
\end{proof}

\begin{corollary}
\label{stability}
The full subcategory 
$\wt{\mrm{Rep}}^{r,\hat{G}_s,J}_{\mrm{tor}}(G_s)$
of $\mrm{Rep}_{\mrm{tor}}(G_s)$
is stable  under subquotients.
\end{corollary}

Let $L$ be as in Section 2, that is,
the completion of an unramified algebraic extension of $K$ with residue field $k_L$.
We prove the following base change lemma.
\begin{lemma}
\label{bc}
Assume that $J\supset u^pI^{[1]}W(R)$ or $L$ is a finite unramified extension of $K$.
If $T$ is an object of $\wt{\mrm{Rep}}^{r,\hat{G}_s,J}_{\mrm{tor}}(G_s)$,
then $T|_{G_{L,s}}$ is an object of $\wt{\mrm{Rep}}^{r,\hat{G}_{L,s},J}_{\mrm{tor}}(G_{L,s})$.
\end{lemma}

\noindent
By an obvious way,
we define a functor $\wtMod^{r,\hat{G}_s}_{/\mfS_{\infty}}\to 
\wtMod^{r,\hat{G}_{L,s}}_{/\mfS_{L,\infty}}$.
The underlying Kisin module of the image of 
$\hat{\mfM}\in \wtMod^{r,\hat{G}_s}_{/\mfS_{\infty}}$ for this functor 
is $\mfM_L=\mfS_L\otimes_{\mfS} \mfM$.
Lemma \ref{bc} immediately follows from the lemma below.
\begin{lemma}
Assume that $J\supset u^pI^{[1]}W(R)$ or $L$ is a finite unramified extension of $K$.
Then the functor 
$\wtMod^{r,\hat{G}_s}_{/\mfS_{\infty}}\to \wtMod^{r,\hat{G}_{L,s}}_{/\mfS_{L,\infty}}$
induces a functor 
$\wtMod^{r,\hat{G}_s,J}_{/\mfS_{\infty}}\to \wtMod^{r,\hat{G}_{L,s},J}_{/\mfS_{L,\infty}}$.
\end{lemma}

\begin{proof}
Let $\hat{\mfM}$ be an object of 
$\wtMod^{r,\hat{G}_s}_{/\mfS_{\infty}}$ and 
let $\hat{\mfM}_L$ be the image of $\hat{\mfM}$ for the functor 
$\wtMod^{r,\hat{G}_s}_{/\mfS_{\infty}}\to \wtMod^{r,\hat{G}_{L,s}}_{/\mfS_{L,\infty}}$.
In the rest of this proof,  to avoid confusions,
we denote the image of $x\in \mfM_L$  in $W(R)\otimes_{\vphi,\mfS_L} \mfM_L$ by $1\otimes x$.
Recall that 
we abuse notations by 
writing $\tau$  for
the pre-image of $\tau\in G_{K,p^{\infty}}$ via the bijection 
$G_{L,p^{\infty}}\simeq G_{K,p^{\infty}}$ of lemma \ref{easylemma}.
Then $\tau^{p^s}$ is a topological generator of $G_{L,s,p^{\infty}}$.
It suffices to show the following: 
if $\hat{\mfM}$ is an object of 
$\wtMod^{r,\hat{G}_s,J}_{/\mfS_{\infty}}$,
then we have $\tau^{p^s}(1\otimes x)-(1\otimes x)\in JW(R)\otimes_{\vphi,\mfS_L} \mfM_L$ 
for any $x\in \mfM_L$. 
Now we suppose $\hat{\mfM}\in \wtMod^{r,\hat{G}_s,J}_{/\mfS_{\infty}}$.
Take any $a\in \mfS_L$ and $x\in \mfM$.
Note that we have 
$\tau^{p^s}(1\otimes ax)-(1\otimes ax)
=\tau^{p^s}(\vphi(a))(\tau^{p^s}(1\otimes x)-(1\otimes x))
+(\tau^{p^s}(\vphi(a))-\vphi(a))(1\otimes x)$
in $W(R)\otimes_{\vphi,\mfS_L} \mfM_L$.
Since $\hat{\mfM}$ is an object of $\wtMod^{r,\hat{G}_s,J}_{/\mfS_{\infty}}$,
we have $\tau^{p^s}(\vphi(a))(\tau^{p^s}(1\otimes x)-(1\otimes x))\in 
JW(R)\otimes_{\vphi,\mfS_L} \mfM_L$.
Therefore, it is enough to show 
$(\tau^{p^s}(\vphi(a))-\vphi(a))(1\otimes x)\in JW(R)\otimes_{\vphi,\mfS_L} \mfM_L$.
This follows from Lemma \ref{cryslem} immediately
in the case where $J\supset u^pI^{[1]}W(R)$.
Next we consider the case where $L$ is a finite unramified extension of $K$.
Let $c_1,\dots ,c_{\ell}\in W(k_L)$ be generators of $W(k_L)$ as a $W(k)$-module.
Then we have $\mfS_L=\sum^{\ell}_{j=1} c_j \mfS$ and thus
we can write $a=\sum^{\ell}_{j=1}a_jc_j$ for some $a_j\in \mfS$.
Hence it suffices to show 
$(\tau^{p^s}(\vphi(a_j))-\vphi(a_j))(1\otimes x)\in JW(R)\otimes_{\vphi,\mfS_L} \mfM_L$
but this in fact immediately follows from the equation
$(\tau^{p^s}(\vphi(a_j))-\vphi(a_j))(1\otimes x)
=(\tau^{p^s}(1\otimes a_jx)-(1\otimes a_jx))-
(\tau^{p^s}(\vphi(a_j))(\tau^{p^s}(1\otimes x)-(1\otimes x)))$.
\end{proof}


\subsection{Full faithfulness theorem for $\wt{\mrm{Rep}}^{r,\hat{G}_s,J}_{\mrm{tor}}(G_s)$}

Our goal in this subsection is to prove the following full faithfulness theorem,
which plays an important roll in our proofs of main theorems.

\begin{theorem}
\label{FFTHM}
Assume that $J\supset u^pI^{[1]}W(R)$ or $k$ is algebraically closed.
If $p^{s+2}/(p-1)\ge c_J>pr/(p-1)$,
then the restriction functor 
$\wt{\mrm{Rep}}^{r,\hat{G}_s,J}_{\mrm{tor}}(G_s)\to \mrm{Rep}_{\mrm{tor}}(G_{\infty})$
is fully faithful.
\end{theorem}

First we give a very rough sketch of the theory of maximal models for Kisin modules
(cf.\ \cite{CL1}). 
For any $\mfM\in \mrm{Mod}^r_{/\mfS_{\infty}}$,
put $\mfM[1/u]=\mfS[1/u]\otimes_{\mfS} \mfM$ and 
denote by $F^r_{\mfS}(\mfM[1/u])$
the (partially) ordered set (by inclusion)
of torsion Kisin modules $\mfN$ of  height $\le r$ which are contained in $\mfM[1/u]$ 
and $\mfN[1/u]=\mfM[1/u]$ as $\vphi$-modules.
The set $F^r_{\mfS}(\mfM[1/u])$ has a greatest element (cf.\ {\it loc}.\ {\it cit}., Corollary 3.2.6).
We denote this element by $\Max^r(\mfM)$.
We say that $\mfM$ is {\it maximal of height $\le r$} (or, {\it maximal} for simplicity)
if it is the greatest element of $F^r_{\mfS}(\mfM[1/u])$. 
The implication $\mfM\mapsto \Max^r(\mfM)$ defines a functor ``$\Max^r$''
from the category $\mrm{Mod}^r_{/\mfS_{\infty}}$ of torsion Kisin modules of height $\le r$ 
into the category $\Max^r_{/\mfS_{\infty}}$ of maximal Kisin modules of height $\le r$.
The category $\Max^r_{/\mfS_{\infty}}$ is abelian (cf.\ {\it loc}.\ {\it cit}., Theorem 3.3.8).
Furthermore, the functor $T_{\mfS}\colon \Max^r_{/\mfS_{\infty}}\to \mrm{Rep}_{\mrm{tor}}(G_{\infty})$,
defined by 
$T_{\mfS}(\mfM)=\mrm{Hom}_{\mfS,\vphi}(\mfM,\mbb{Q}_p/\mbb{Z}_p\otimes_{\mbb{Z}_p} W(R))$,
is exact and fully faithful (cf.\ {\it loc}.\ {\it cit}., Corollary 3.3.10). 
It is not difficult to check that
$T_{\mfS}(\Max^r(\mfM))$ is canonically isomorphic to $T_{\mfS}(\mfM)$ as representations of $G_{\infty}$
for any torsion Kisin module $\mfM$ of height $\le r$. 

\begin{definition}[{\rm \cite[Section 3.6.1]{CL1}}]
\label{Def1}
Let $d$ be a positive integer.
Let $\mfn=(n_i)_{i\in \mbb{Z}/d\mbb{Z}}$ be a sequence
of non-negative integers of smallest period $d$.
We define a torsion Kisin module $\mfM(\mfn)$ as below:
\begin{itemize}
\item as a $\ku$-module, $\mfM(\mfn)=\bigoplus_{i\in \mbb{Z}/d\mbb{Z}} \ku e_i$;
\item for all $i\in \mbb{Z}/d\mbb{Z}$, $\vphi(e_i)=u^{n_i}e_{i+1}$.
\end{itemize}
\end{definition}
We denote by $\mcal{S}^r_{\mrm{max}}$ the set of sequences $\mfn=(n_i)_{i\in \mbb{Z}/d\mbb{Z}}$
of integers $0\le n_i\le \mrm{min}\{er, p-1\}$ with smallest period $d$ for some integer $d$
except the constant sequence with value $p-1$ (if necessary).  
By definition, we see that 
$\mfM(\mfn)$ is of height $\le r$ for any $\mfn\in \mcal{S}^r_{\mrm{max}}$.
Putting $r_0=\mrm{max}\{r'\in \mbb{Z}_{\ge 0};e(r'-1)<p-1 \}$,
we also see that 
$\mfM(\mfn)$ is of height $\le r_0$ for any $\mfn\in \mcal{S}^r_{\mrm{max}}$.
It is known that
$\mfM(\mfn)$ is 
maximal for any $\mfn\in \mcal{S}^r_{\mrm{max}}$ (\cite[Proposition 3.6.7]{CL1}).
If $k$ is algebraically closed,
then $\mfM(\mfn)$ is 
simple in $\Max^r_{/\mfS_{\infty}}$  for any $\mfn\in \mcal{S}^r_{\mrm{max}}$ 
(cf.\ {\it loc.\ cit.}, Propositions 3.6.7 and 3.6.12)
and furthermore, the converse holds;
any simple object in $\Max^r_{/\mfS_{\infty}}$ is of the form
$\mfM(\mfn)$ for some $\mfn\in \mcal{S}^r_{\mrm{max}}$
(cf.\ {\it loc.\ cit.}, Propositions 3.6.8 and 3.6.12).


\begin{lemma}
\label{Lem1}
Assume that $p^{s+2}/(p-1)\ge c_J$.
Let $d$ be a positive integer.
Let $\mfn=(n_i)_{i\in \mbb{Z}/d\mbb{Z}}$ be a sequence
of non-negative integers  of smallest period $d$.
If $\mfM(\mfn)$ is of height $\le r$,
then $\mfM(\mfn)$ has a structure of an object of 
$\wtMod^{r,\hat{G}_s,J}_{/\mfS_{\infty}}$.
\end{lemma}
\begin{proof}
Choose any $(p^d-1)$-th root $\eta\in R$ 
of $\underline{\e}$.
Since $[\eta]\cdot \mrm{exp}(t/(p^d-1))$ is a $(p^d-1)$-th 
root of unity, 
it is of the form $[a]$ for some $a\in \mbb{F}^{\times}_{p^d}$.
Replacing  $\eta a^{-1}$ with $\eta$,
we obtain $[\eta]=\mrm{exp}(-t/(p^d-1))\in \whR^{\times}$. 
Put $x_i=[\eta]^{m_i}\in \whR^{\times}$ and  
$\bar{x}_i=\eta^{m_i}\in (\whR/p\whR)^{\times}\subset R^{\times}$ 
for any $i\in \mbb{Z}/d\mbb{Z}$,
where $m_i=\sum^{d-1}_{j=0}n_{i+j}p^{d-j}$.
We see that $x_i-1$ is contained in $I_+\whR$.
In the rest of this proof, to avoid confusions,
we denote the image of $x\in  \mfM(\mfn)$ in 
$\whR_s\otimes_{\vphi,\mfS} \mfM(\mfn)\subset R\otimes_{\vphi,k[\![u]\!]} \mfM(\mfn)$ 
by $1\otimes x$.
Now we define a $\hat{G}_s$-action on $\whR_s\otimes_{\vphi,\mfS} \mfM(\mfn)$ by
$\tau^{p^s}(1\otimes e_i):=x^{p^s}_i(1\otimes e_i)$ for the basis 
$\{e_i\}_{i\in \mbb{Z}/d\mbb{Z}}$ 
of $\mfM(\mfn)$
as in Definition \ref{Def1}.
It is not difficult to check that 
$\mfM(\mfn)$ has a structure of an object of 
$\wtMod^{r,\hat{G}_s}_{/\mfS_{\infty}}$
via this $\hat{G}_s$-action; we denote it by $\hat{\mfM}(\mfn)$.
It suffices to prove that 
$\hat{\mfM}(\mfn)$ is in fact an object of 
$\wtMod^{r,\hat{G}_s,J}_{/\mfS_{\infty}}$.
Recall that $v_R$ is the valuation of $R$ normalized such that $v_R(\underline{\pi})=1/e$.
Define $\tilde{\mft}=\mft\ \mrm{mod}\ pW(R)$ an element of $R$.
We denote by $v_p$ the usual $p$-adic valuation normalized by $v_p(p)=1$.
Note that we have $v_R(\underline{\e}-1)=p/(p-1)$ and 
$v_R(\tilde{\mft})=1/(p-1)$
(here, the latter equation follows from the relation
$\vphi(\mft)=pE(0)^{-1}E(u)\mft$).
We see that 
$$
v_R(\bar{x}^{p^s}_i-1)=p^{s+v_p(m_i)}\cdot p/(p-1)\ge p^{s+2}/(p-1).
$$
Since $p^{s+2}/(p-1)\ge c_J$ and the image of $J$ in $R$
is $\mfm_R^{\ge c_J}$, we obtain
$$
\tau^{p^s}(1\otimes e_i)-(1\otimes e_i)\in 
\mfm_R^{\ge c_J}R\otimes_{\vphi,k[\![u]\!]} \mfM(\mfn)
\simeq JW(R)\otimes_{\vphi,\mfS} \mfM(\mfn).
$$
Finally we have to show that $\tau^{p^s}(1\otimes ae_i)-(1\otimes ae_i)\in 
\mfm_R^{\ge c_J}R\otimes_{\vphi,k[\![u]\!]} \mfM(\mfn)$ 
for any $a\in k[\![u]\!]$.
Since $\tau^{p^s}(1\otimes ae_i)-(1\otimes ae_i)=\tau^{p^s}(\vphi(a))
(\tau^{p^s}(1\otimes e_i)-(1\otimes e_i))+
(\tau^{p^s}(\vphi(a))-\vphi(a))(1\otimes e_i)$,
it suffices to show $\tau^{p^s}(\vphi(a))-\vphi(a)\in \mfm_R^{\ge c_J}$.
Write $\vphi(a)=\sum_{i\ge 0}a_iu^{pi}$ for some $a_i\in k$.
Then we have $\tau^{p^s}(\vphi(a))-\vphi(a)=
\sum_{i\ge 1}a_i(\underline{\e}^{p^{s+1}i}-1)u^{pi}$.
Since we have
$$
v_R((\underline{\e}^{p^{s+1}i}-1)u^{pi})
=p^{s+1}v_R(\underline{\e}^i-1)+v_R(u^{pi})
> p^{s+2}/(p-1)\ge c_J
$$
for any $i\ge 1$, we have done.
\end{proof}

Recall that $r_0=\mrm{max}\{r'\in \mbb{Z}_{\ge 0}; e(r'-1)<p-1\}$.
Put $r_1:=\mrm{min}\{r,r_0\}$.

\begin{corollary}
\label{Cor1}
Assume that $p^{s+2}/(p-1)\ge c_J$.
If $\mfn\in \mcal{S}^r_{\mrm{max}}$, 
then 
$\mfM(\mfn)$ has a structure of an object of 
$\wtMod^{r',\hat{G}_s,J}_{/\mfS_{\infty}}$
for any $r'\ge r_1$.
Furthermore, if $c_J>pr_1/(p-1)$,
it is uniquely determined.
We denote this object by $\hat{\mfM}(\mfn)$.
\end{corollary}
\begin{proof}
We should remark that  $\mfM(\mfn)$ is of height $\le r_1$ 
for any $\mfn\in \mcal{S}^r_{\mrm{max}}$.
The uniqueness assertion follows from 
Proposition \ref{FFTHMMOD}. 
\end{proof}

\begin{lemma}
\label{tameres}
The functor from tamely ramified $\mbb{Z}_p$-representations of $G_K$
to $\mbb{Z}_p$-representations of $G_{\infty}$,
obtained by restricting the action of $G_K$ to $G_{\infty}$, is fully faithful.
\end{lemma}

\begin{proof}
The result immediately follows from the fact that 
$G_K$ is topologically generated by $G_{\infty}$ and the wild inertia subgroup of $G_K$.
\end{proof}
\noindent
We remark that any semi-simple $\mbb{F}_p$-representation of $G_K$ is automatically tame.

\begin{lemma}
\label{FFLEM}
Assume that $J\supset u^pI^{[1]}W(R)$ or $k$ is algebraically closed.
Let $T\in \mrm{Rep}_{\mrm{tor}}(G_s)$ and 
$T'\in \wt{\mrm{Rep}}^{r,\hat{G}_s,J}_{\mrm{tor}}(G_s)$.
Suppose that $T$ is tame, $pT=0$ and $T|_{G_{\infty}}\simeq T_{\mfS}(\mfM)$ 
for some $\mfM\in \Mod^r_{/\mfS_{\infty}}$. 
Furthermore,
we suppose 
$p^{s+2}/(p-1)\ge c_J>pr/(p-1)$.
Then all $G_{\infty}$-equivariant homomorphisms $T\to T'$ are in fact $G_s$-equivariant.
\end{lemma}

\begin{proof}
Let $L$ be the completion of the 
maximal unramified extension $K^{\mrm{ur}}$ of $K$.
By identifying $G_L$ with $G_{K^{\mrm{ur}}}$, 
we may regard $G_L$ as a subgroup of $G_K$.
Note that $L_{(s)}=K_{(s)}L$ is the completion of the maximal unramified extension of $K_{(s)}$,
and $G_s$ is topologically generated by $G_{L,s}$ and $G_{\infty}$.
Consider the following commutative diagram:
\begin{center}
$\displaystyle \xymatrix{
\mrm{Hom}_{G_{L,s}}(T,T')\ar@{^{(}->}[rr] & &   
\mrm{Hom}_{G_{L,\infty}}(T,T') \\
\mrm{Hom}_{G_s}(T,T') \ar@{^{(}->}[u] \ar@{^{(}->}[rr] & &   
\mrm{Hom}_{G_{\infty}}(T,T'). \ar@{^{(}->}[u]
}$
\end{center}
\noindent
Since $T'|_{G_{L,s}}$ is contained in 
$\wt{\mrm{Rep}}^{r,\hat{G}_{L,s},J}_{\mrm{tor}}(G_{L,s})$
if $J\supset u^pI^{[1]}W(R)$ (cf.\ Lemma \ref{bc}),
the above diagram allows us to reduce a proof to the case where $k$ is algebraically closed.
In the rest of this proof, we assume that $k$ is algebraically closed. 
Under this assumption, an $\mbb{F}_p$-representation of $G_s$ is 
tame if and only if it is semi-simple by Maschke's theorem.
Thus we may also assume that $T$ is irreducible
(here, we remark that any subquotient of $T$ is tame and,
also remark that the essential image of 
$T_{\mfS}\colon \Mod^r_{/\mfS_{\infty}}\to \mrm{Rep}_{\mrm{tor}}(G_{\infty})$
is stable under subquotients in $\mrm{Rep}_{\mrm{tor}}(G_{\infty})$).
We claim that 
$T|_{G_{\infty}}$ is also irreducible.
If not, there exists a non-zero irreducible 
$\mbb{F}_p[G_{\infty}]$-submodule $W$ of $T|_{G_{\infty}}$.
Let $K^{\mrm{t}}_{(s)}$ be the maximal tamely ramified extension of $K_{(s)}$
and $I_{p,s}:=\mrm{Gal}(\overline{K}/K^{\mrm{t}}_{(s)})$ the wild inertia subgroup of $G_s$.
We see that $K^{\mrm{t}}_{(s)}\cap K_{\infty}=K_{(s)}$. 
Since $G_{\infty}\cap I_{p,s}$ acts on $W$ trivially,
the $G_{\infty}$-action on $W$ extends to $G_s$ via the composition map
$G_s\twoheadrightarrow \mrm{Gal}(K^{\mrm{t}}_{(s)}/K_{(s)})
\simeq G_{\infty}/(G_{\infty}\cap I_{p,s})$.
Thus we can regard $W$ as an irreducible $\mbb{F}_p[G_s]$-module.
By Lemma \ref{tameres}, we see that $W$ is a sub $\mbb{F}_p[G_s]$-module of $T$. 
This contradicts the irreducibility of $T$ and the claim follows.

By the assumption on $T$, we have
$T|_{G_{\infty}}\simeq T_{\mfS}(\mfM)\simeq T_{\mfS}(\Max^r(\mfM))$ 
for some $\mfM\in \Mod^r_{/\mfS_{\infty}}$.
Since $T|_{G_{\infty}}$ is irreducible and 
$T_{\mfS}\colon \Max^r_{/\mfS_{\infty}}\to \mrm{Rep}_{\mrm{tor}}(G_{\infty})$
is exact and fully faithful,
we know that $\Max^r(\mfM)$ is a simple object in 
the abelian category $\Max^r_{/\mfS_{\infty}}$.
Therefore, since $k$ is algebraically closed,
we have $\Max^r(\mfM)\simeq \mfM(\mfn)$ for some
$\mfn\in \mcal{S}^r_{\mrm{max}}$ (cf.\ \cite[Propositions 3.6.8 and 3.6.12]{CL1}).
Let $\hat{\mfM}(\mfn)$  
be the object of $\wtMod^{r,\hat{G},J}_{/\mfS_{\infty}}$ 
as in Corollary \ref{Cor1}.
We recall that $T_{\mfS}(\mfM(\mfn))$ is isomorphic to 
$\hat{T}_s(\hat{\mfM}(\mfn))|_{G_{\infty}}$ (see Theorem \ref{Thm1} (1)),
and hence we have an isomorphism 
$T|_{G_{\infty}}\simeq \hat{T}_s(\hat{\mfM}(\mfn))|_{G_{\infty}}$.
Here, we note that $T$ and $\hat{T}_s(\hat{\mfM}(\mfn))$ are irreducible 
as representations of $G_s$ (cf.\ \cite[Theorem 3.6.11]{CL1}).  
Applying Lemma \ref{tameres} again,
we obtain an isomorphism  $T\simeq \hat{T}_s(\hat{\mfM}(\mfn))$
as representations of $G_s$.
On the other hand, we can take 
$\hat{\mfM}'=(\mfM',\vphi,\hat{G}_s)\in  \wtMod^{r,\hat{G}_s,J}_{/\mfS_{\infty}}$
such that $T'\simeq \hat{T}_s(\hat{\mfM}')$.
We consider the following commutative diagram:
\begin{center}
$\displaystyle \xymatrix{
\mrm{Hom}_{G_s}(T,T')\ar@{^{(}->}[rr] & &   
\mrm{Hom}_{G_{\infty}}(T,T') \\
\mrm{Hom}(\hat{\mfM}',\hat{\mfM}(\mfn)) \ar^{\hat{T}_s}[u] \ar^{\mrm{forgetful}\ }[r] &
\mrm{Hom}(\mfM',\mfM(\mfn)) \ar^{\Max^r\quad \ \ }[r] & \mrm{Hom}(\Max^r(\mfM'),\mfM(\mfn)). 
\ar^{T_{\mfS}}[u]. 
}$
\end{center}
Here, $\mrm{Hom}(\hat{\mfM}',\hat{\mfM}(\mfn))$ 
(resp.\ $\mrm{Hom}(\mfM',\mfM(\mfn))$, resp.\ $\mrm{Hom}(\Max^r(\mfM'),\mfM(\mfn))$)
is the set of morphisms 
$\hat{\mfM}'\to \hat{\mfM}(\mfn)$ 
(resp.\ $\mfM'\to \mfM(\mfn)$, resp.\ $\Max^r(\mfM')\to \mfM(\mfn)$)
in $\wtMod^{r_1,\hat{G}_s,J}_{/\mfS_{\infty}}$
(resp.\ $\Mod^{r_1}_{/\mfS_{\infty}}$, resp.\ $\Max^r_{/\mfS_{\infty}}$).
The first bottom horizontal arrow is bijective by Theorem \ref{Thm1} (3) and 
so is the second (this follows from the fact that $\mfM(\mfn)$ is maximal by \cite[Proposition 3.6.7]{CL1}).
Since the right vertical arrow is bijective,
the top horizontal arrow must be bijective. 
\end{proof}

Now we are ready to prove Theorem \ref{FFTHM}.

\begin{proof}[Proof of Theorem \ref{FFTHM}]
Let $T$ and $T'$ be objects of $\wt{\mrm{Rep}}^{r,\hat{G}_s,J}_{\mrm{tor}}(G_s)$.
Take any Jordan-H\"ollder sequence $0=T_0\subset T_1\subset \cdots \subset T_n=T$ 
of $T$ in $\mrm{Rep}_{\mrm{tor}}(G_s)$.
By Corollary \ref{stability},
we know that $T_i$ and $T_i/T_{i-1}$ are contained in $\wt{\mrm{Rep}}^{r,\hat{G}_s,J}_{\mrm{tor}}(G_s)$
for any $i$.
By Corollary \ref{stability} again,
the category $\wt{\mrm{Rep}}^{r,\hat{G}_s,J}_{\mrm{tor}}(G_s)$ 
is an exact category in the sense of Quillen (\cite[Section 2]{Qu}).
Hence short exact sequences in $\wt{\mrm{Rep}}^{r,\hat{G}_s,J}_{\mrm{tor}}(G_s)$ give rise
to exact sequences of Hom's and Ext's in the usual way.
(This property holds for any exact category.)
On the other hand, by Lemma \ref{FFLEM},
if an exact sequence 
$0\to T'\to V\to T_i/T_{i-1}\to 0$ in $\wt{\mrm{Rep}}^{r,\hat{G}_s,J}_{\mrm{tor}}(G_s)$
splits as representation of $G_{\infty}$, then 
it splits as a sequence of representations of $G_s$.
Therefore, comparing exact sequences 
of Hom's and Ext's arising from 
$0\to T_{i-1}\to T_i\to T_i/T_{i-1}\to 0$ in the category
$\wt{\mrm{Rep}}^{r,\hat{G}_s,J}_{\mrm{tor}}(G_s)$
with that in the category $\mrm{Rep}_{\mrm{tor}}(G_{\infty})$,
we obtain the following implication (here, use Lemma \ref{FFLEM} again):
if we have $\mrm{Hom}_{G_s}(T_{i-1}, T')=\mrm{Hom}_{G_{\infty}}(T_{i-1}, T')$,
then it gives the equality $\mrm{Hom}_{G_s}(T_i, T')=\mrm{Hom}_{G_{\infty}}(T_i, T')$.
Hence a d\'evissage argument works and the desired full faithfulness follows. 
\end{proof}


\subsection{Proof of Theorem \ref{Main1}}

Now we are ready to prove our main theorems.
First we prove Theorem \ref{Main1}.

Recall that $\mrm{Rep}^{r,\mrm{ht},\mrm{pcris}(s)}_{\mrm{tor}}(G_K)$
is the category of torsion $\mbb{Z}_p$-representations $T$ of $G_K$
which satisfy the following:
there exist free $\mbb{Z}_p$-representations 
$L$ and $L'$ of $G_K$, of height $\le r$, such that
\begin{itemize}
\item  $L|_{G_s}$ is a subrepresentation of $L'|_{G_s}$.
Furthermore, $L|_{G_s}$ and $L'|_{G_s}$ are lattices in some 
crystalline $\mbb{Q}_p$-representation of $G_s$ with Hodge-Tate weights in $[0,r]$;
\item  $T|_{G_s} \simeq (L'|_{G_s})/(L|_{G_s})$.
\end{itemize}

\noindent
We apply our arguments given in previous subsections with the following $J$: 
$$
J=u^pI^{[1]}W(R)=u^p\vphi(\mft)W(R).
$$
Then we have $c_J=p/e+p/(p-1)$ and thus the inequalities 
$p^{s+2}/(p-1)\ge c_J>pr/(p-1)$
are satisfied if $e(r-1)<p-1$.
Therefore, 
Theorem \ref{Main1} is an easy consequence of the following proposition and Theorem \ref{FFTHM}.
\begin{proposition}
\label{proofMain1}
If $T$ is an object of $\mrm{Rep}^{r,\mrm{ht},\mrm{pcris}(s)}_{\mrm{tor}}(G_K)$,
then $T|_{G_s}$ is contained in $\wt{\mrm{Rep}}^{r,\hat{G}_s,J}_{\mrm{tor}}(G_s)$.
\end{proposition}

\begin{proof}
Take free $\mbb{Z}_p$-representations 
$L$ and $L'$ of $G_K$, of height $\le r$, such that
\begin{itemize}
\item  $L|_{G_s}$ is a subrepresentation of $L'|_{G_s}$.
Furthermore, $L|_{G_s}$ and $L'|_{G_s}$ are lattices in some 
crystalline $\mbb{Q}_p$-representation of $G_s$ with Hodge-Tate weights in $[0,r]$;
\item  $T|_{G_s} \simeq (L'|_{G_s})/(L|_{G_s})$.
\end{itemize}
By Theorem \ref{Thm1} (1), there exists an injection $\hat{\mfL}'\hookrightarrow \hat{\mfL}$ of 
$(\vphi,\hat{G}_s)$-modules over $\mfS_s$ which corresponds to
the injection $L|_{G_s}\hookrightarrow L'|_{G_s}$. 
On the other hand, there exist $\mfN$ and $\mfN'$ in $\Mod^r_{/\mfS}$ such that 
$T_{\mfS}(\mfN)\simeq L|_{G_{\infty}}$ and $T_{\mfS}(\mfN')\simeq L'|_{G_{\infty}}$.
Then Proposition \ref{Kisinfunctor} (1) implies that $\mfS_s\otimes_{\mfS} \mfN\simeq \mfL$
and $\mfS_s\otimes_{\mfS} \mfN'\simeq \mfL'$ as $\vphi$-modules over $\mfS_s$.
Therefore,
we see that $\mfN$ and $\mfN'$ have structures of objects of $\wtMod^{r,\hat{G}_s}_{/\mfS}$;
denote them by $\hat{\mfN}$ and $\hat{\mfN}'$, respectively.
Since the functor $\wtMod^{r,\hat{G}_s}_{/\mfS}\to \Mod^{r,\hat{G}_s}_{/\mfS_s}$
is fully faithful (cf.\ Section \ref{vardef}), the injection $\hat{\mfL}'\hookrightarrow \hat{\mfL}$
descends to an injection $\hat{\mfN}'\hookrightarrow \hat{\mfN}$.
Now we put $\mfM=\mfN/\mfN'$. 
Since $(L'|_{G_s})/(L|_{G_s})$ is killed by a power of $p$, 
it is an object of $\Mod^r_{/\mfS_{\infty}}$.
We equip a $\hat{G}_s$-action with $\whR_s\otimes_{\vphi,\mfS} \mfM$ 
by a natural isomorphism $\whR_s\otimes_{\vphi,\mfS} \mfM\simeq 
(\whR_s\otimes_{\vphi,\mfS} \mfN)/(\whR_s\otimes_{\vphi,\mfS} \mfN')$.
Then we see that $\mfM$ has a structure of an object of 
$\wtMod^{r,\hat{G}_s}_{/\mfS_{\infty}}$; denote it by $\hat{\mfM}$.
Moreover, Theorem \ref{cris} implies that 
$\hat{\mfM}$ is in fact contained in  
$\wtMod^{r,\hat{G}_s,J}_{/\mfS_{\infty}}$.
By a similar argument to the proof of  Lemma 3.1.4 of  \cite{CL2},
we have an exact sequence 
$0\to \hat{T}_s(\hat{\mfN})\to \hat{T}_s(\hat{\mfN}')\to \hat{T}_s(\hat{\mfM})\to 0$
in $\mrm{Rep}_{\mrm{tor}}(G_s)$, which is isomorphic to 
$0\to L|_{G_s}\to L'|_{G_s}\to T|_{G_s}\to 0$. 
This finishes a proof.
\end{proof}


\subsection{Proof of Theorem \ref{Main2}}

We give a proof of Theorem \ref{Main2}.
If $s\ge n-1$, then we put
$$
J=u^pI^{[p^{s-n+1}]}W(R)=u^p\vphi(\mft)^{p^{s-n+1}}W(R).
$$
Note that we have $c_J=p/e+p^{s-n+2}/(p-1)$ and thus the inequalities 
$p^{s+2}/(p-1)\ge c_J>pr/(p-1)$
are satisfied
if $s>n-1+\mrm{log}_p(r-e/(p-1))$.
\begin{proposition}
\label{Main2Lem}
Suppose $s\ge n-1$.
If $T$ is an object of $\mrm{Rep}^{r,\mrm{cris}}_{\mrm{tor}}(G_K)$ which is killed by $p^n$,
then $T|_{G_s}$ is contained in $\wt{\mrm{Rep}}^{r,\hat{G}_s,J}_{\mrm{tor}}(G_s)$.
\end{proposition}

\begin{proof}
Let $L$ be an object of $\mrm{Rep}^{r,\mrm{cris}}_{\mbb{Z}_p}(G_K)$.
Take a $(\vphi,\hat{G})$-module $\hat{\mfL}$ over $\mfS$ such that 
$L\simeq \hat{T}(\hat{\mfL})$.
It is known that $(\tau-1)^i(x)\in u^pI^{[i]}W(R)\otimes_{\vphi,\mfS} \mfL$ for any $i\ge 1$ and any $x\in \mfL$
(cf.\ the latter half part of the proof of \cite[Proposition 4.7]{GLS}).
Take any $x\in \mfL$.
Since $(\tau^{p^s}-1)(x)=\sum^{p^s}_{i=1}\binom{p^s}{i}(\tau-1)^i(x)$,
we obtain that 
\begin{equation}
\label{relat}
(\tau^{p^s}-1)(x)\in \sum^{p^s}_{i=1} p^{s-v_p(i)}u^pI^{[i]}W(R)\otimes_{\vphi,\mfS} \mfL.
\end{equation}
Now let $T$ be an object of $\mrm{Rep}^{r,\mrm{cris}}_{\mrm{tor}}(G_K)$ which is killed by $p^n$.
Take an exact sequence $(R)\colon 0\to L_1\to L_2\to T\to 0$ of $\mbb{Z}_p$-representations of $G_K$
with $L_1,L_2\in \mrm{Rep}^{r,\mrm{cris}}_{\mbb{Z}_p}(G_K)$.
By Theorem 3.1.3 and Lemma 3.1.4 of  \cite{CL2},
there exists an exact sequence 
$(M)\colon 0\to \hat{\mfL}_2\to \hat{\mfL}_1\to \hat{\mfM}\to 0$
of $(\vphi,\hat{G})$-modules over $\mfS$ such that 
$\hat{T}((M))\simeq (R)$.
By (\ref{relat}),
we see that
$$
(\tau^{p^s}-1)(x)\in \sum^{p^s}_{i=1} p^{s-v_p(i)}u^pI^{[i]}W(R)\otimes_{\vphi,\mfS} \mfM
$$
for any $x\in \mfM$.
Since $\mfM$ is killed by $p^n$ and $s\ge n-1$,
we have 
\begin{align*}
\sum^{p^s}_{i=1} p^{s-v_p(i)}u^pI^{[i]}W(R)\otimes_{\vphi,\mfS} \mfM
&= \sum_{i=1,\dots ,p^s, s-v_p(i)<n} p^{s-v_p(i)}u^pI^{[i]}W(R)\otimes_{\vphi,\mfS} \mfM\\
&= \sum^{n-1}_{\ell=0} p^{\ell}u^pI^{[p^{s-\ell}]}W(R)\otimes_{\vphi,\mfS} \mfM\\
& \subset u^pI^{[p^{s-n+1}]}W(R)\otimes_{\vphi,\mfS} \mfM.
\end{align*}
Therefore, we obtained the desired result.
\end{proof}

\begin{proof}[Proof of Theorem \ref{Main2}]
By Corollary \ref{FFTHMtorcris},
we may suppose $\mrm{log}_p(r-(p-1)/e)\ge 0$, that is, $e(r-1)\ge p-1$ .
Suppose $s>  n-1+\mrm{log}_p(r-(p-1)/e)$.
Note that the condition $s\ge n-1$ is now satisfied.
Let $T$ and $T'$ be as in the statement of Theorem \ref{Main2}.
Let $f\colon T\to T'$ be a $G_{\infty}$-equivariant homomorphism. 
Denote by $L$ the completion of $K^{\mrm{ur}}$
and identify $G_L$ with the inertia subgroup of $G_K$. 
We note that $T|_{G_L}$ and $T'|_{G_L}$ are object of 
$\mrm{Rep}^{r,\mrm{cris}}_{\mrm{tor}}(G_L)$.
By Proposition \ref{Main2Lem},
$T|_{G_{L,s}}$ and $T'|_{G_{L,s}}$ are objects of 
$\wt{\mrm{Rep}}^{r,\hat{G}_{L,s},J}_{\mrm{tor}}(G_{L,s})$.
Hence we have that $f$ is $G_{L,s}$-equivariant by Theorem \ref{FFTHM}.
Since $G_s$ is topologically generated by $G_{L,s}$ and $G_{\infty}$, 
we see that $f$ is $G_s$-equivariant.
\end{proof}


\subsection{Galois equivariance for torsion semi-stable representations}
\label{torsemi}

In this subsection, we prove a Galois equivariance theorem for 
torsion semi-stable representations. 
A torsion $\mbb{Z}_p$-representation $T$ of $G_K$ 
is {\it torsion semi-stable 
with Hodge-Tate weights in $[0,r]$}
if it can be written as the quotient of lattices in some semi-stable 
$\mbb{Q}_p$-representation of $G_K$ with Hodge-Tate weights in $[0,r]$. 
We denote by $\mrm{Rep}^{r, \mrm{st}}_{\mrm{tor}}(G_K)$
the category of them.
Note that $\mrm{Rep}^{0, \mrm{st}}_{\mrm{tor}}(G_K)=
\mrm{Rep}^{0, \mrm{cris}}_{\mrm{tor}}(G_K)$.
Similar to Theorem \ref{Main2},
we show the following, which is the main result of this subsection.

\begin{theorem}
\label{Main3}
Suppose that $s> n-1 + \mrm{log}_pr$.
Let $T$ and $T'$ be objects of $\mrm{Rep}^{r, \mrm{st}}_{\mrm{tor}}(G_K)$
which are killed by $p^n$.
Then any $G_{\infty}$-equivariant homomorphism $T\to T'$ is in fact $G_s$-equivariant.
\end{theorem}

If $s\ge n-1$, then we put
$$
J=I^{[p^{s-n+1}]}W(R)=\vphi(\mft)^{p^{s-n+1}}W(R).
$$
Then we have $c_J=p^{s-n+2}/(p-1)$.
To show Theorem \ref{Main3},
we use similar arguments to those in the proof of Theorem \ref{Main2}.

\begin{proposition}
\label{Main3Lem}
Suppose $s\ge n-1$.
If $T$ is an object of $\mrm{Rep}^{r,\mrm{st}}_{\mrm{tor}}(G_K)$ 
which is killed by $p^n$,
then $T|_{G_s}$ is contained in $\wt{\mrm{Rep}}^{r,\hat{G}_s,J}_{\mrm{tor}}(G_s)$.
\end{proposition}
\begin{proof}
Let $L$ be a lattice in a semi-stable $\mbb{Q}_p$-representation of $G_K$
with Hodge-Tate weights in $[0,r]$.
Take a $(\vphi,\hat{G})$-module $\hat{\mfL}$ over $\mfS$ such that 
$L\simeq \hat{T}(\hat{\mfL})$.
It is known that $(\tau-1)^i(x)\in I^{[i]}W(R)\otimes_{\vphi,\mfS} \mfL$ for any $i\ge 1$ and any $x\in \mfL$
(cf.\ the proof of \cite[Proposition 2.4.1]{Li4}).
Thus the same proof proceeds as that of Proposition \ref{Main2Lem}.
\end{proof}

\begin{proof}[Proof of Theorem \ref{Main3}]
We have the equality 
$\mrm{Rep}^{0, \mrm{st}}_{\mrm{tor}}(G_K)=\mrm{Rep}^{0, \mrm{cris}}_{\mrm{tor}}(G_K)$
and thus Theorem \ref{Main2} for $r=0$ is an easy consequence of Corollary \ref{FFTHMtorcris}.
Hence we may assume $r\ge 1$.
The rest of a proof is similar to the proof of Theorem \ref{Main2}.
\end{proof}

\subsection{Some consequences}
\label{consequences}

In this subsection, we generalize some results proved in Section 3.4 of \cite{Br3}.
First of all, we show the following elementary lemma,
which should be well-known to experts, 
but we include a proof here for the sake of completeness.

\begin{lemma} 
\label{stability'} 
The full subcategories $\mrm{Rep}^{r,\mrm{cris}}_{\mrm{tor}}(G_K)$ 
and $\mrm{Rep}^{r,\mrm{st}}_{\mrm{tor}}(G_K)$ of 
$\mrm{Rep}_{\mrm{tor}}(G_K)$ are stable under formation of 
subquotients, direct sums and the association $T\mapsto T^{\vee}(r)$. 
Here $T^{\vee}=\mrm{Hom}_{\mbb{Z}_p}(T,\mbb{Q}_p/\mbb{Z}_p)$ is the dual representation of $T$. 
\end{lemma}

\begin{proof}
We prove the statement only for $\mrm{Rep}^{r,\mrm{cris}}_{\mrm{tor}}(G_K)$.
Let $T\in \mrm{Rep}^{r,\mrm{cris}}_{\mrm{tor}}(G_K)$ be killed by $p^n$ for some $n>0$.
Assertions for quotients and direct sums are clear. 
We prove that $T^{\vee}(r)$ is contained in $\mrm{Rep}^{r,\mrm{cris}}_{\mrm{tor}}(G_K)$.
There exist lattices $L_1\subset L_2$ 
in some crystalline $\mbb{Q}_p$-representation of $G_K$
and an exact sequence 
$0\to L_1\to L_2\to T\to 0$ of $\mbb{Z}_p[G_K]$-modules.
This exact sequence induces an exact sequence 
$0\to T\to L_1/p^nL_1\to L_2/p^nL_2\to T\to 0$ of finite 
$\mbb{Z}_p[G_K]$-modules.
By duality,
we obtain an exact sequence 
$0\to T^{\vee}\to (L_2/p^nL_2)^{\vee}\to 
(L_1/p^nL_1)^{\vee}\to T^{\vee}\to 0$ of finite 
$\mbb{Z}_p[G_K]$-modules.
Then we obtain a $G_K$-equivariant surjection $L_1^{\vee}\twoheadrightarrow T^{\vee}$
by the composite
$L_1^{\vee}\twoheadrightarrow L_1^{\vee}/p^nL_1^{\vee}\overset{\sim}{\to} 
(L_1/p^nL_1)^{\vee}\twoheadrightarrow T^{\vee}$ of natural maps
(here, for any free $\mbb{Z}_p$-representation $L$ of $G_K$,
$L^{\vee}:=\mrm{Hom}_{\mbb{Z}_p}(L,\mbb{Z}_p)$ stands for the dual of $L$).
Therefore, we obtain $L_1^{\vee}(r)\twoheadrightarrow T^{\vee}(r)$
and thus $T^{\vee}(r)\in \mrm{Rep}^{r,\mrm{cris}}_{\mrm{tor}}(G_K)$.
Finally, we prove the stability assertion for subobjects.
Let $T'$ be a $G_K$-stable submodule of $T$.  
We have a $G_K$-equivariant surjection
$f\colon L_1^{\vee}\twoheadrightarrow T^{\vee}\twoheadrightarrow (T')^{\vee}$.
Let $L'_2$ be a free $\mbb{Z}_p$-representation of $G_K$
such that its dual is the kernel of $f$.
We have an exact sequence $0\to (L'_2)^{\vee}\to L^{\vee}_1\overset{f}{\to} (T')^{\vee}\to 0$
of $\mbb{Z}_p[G_K]$-modules.
Repeating the construction of the surjection $L_1^{\vee}\twoheadrightarrow T^{\vee}$,
we obtain a $G_K$-equivariant surjection $L'_2=(L'_2)^{\vee \vee}\twoheadrightarrow (T')^{\vee \vee}=T'$
and thus we have $T'\in \mrm{Rep}^{r,\mrm{cris}}_{\mrm{tor}}(G_K)$.
\end{proof}

In the case where $r=1$,
the assertion (1) of the following corollary was shown in Theorem 3.4.3 of \cite{Br3}.
\begin{corollary}
\label{imagestable}
Let $T$ be an object of $\mrm{Rep}^{r,\mrm{cris}}_{\mrm{tor}}(G_K)$ which is killed by $p^n$ for some $n>0$.
Let $T'$ be a $G_{\infty}$-stable subquotient of $T$. 

\noindent
$(1)$ If $e(r-1)<p-1$, then $T'$ is $G_K$-stable $($with respect to $T)$.

\noindent
$(2)$ If $s>n-1+\mrm{log}_p(r-(p-1)/e)$, 
then $T'$ is $G_s$-stable $($with respect to $T)$. 
\end{corollary}

\begin{proof}
By the duality assertion of Lemma \ref{stability'},
it is enough to show the case where 
$T'$ is a $G_{\infty}$-stable submodule of $T$.
Take any sequence
$T'=T_0\subset T_1\subset \cdots \subset T_m=T$
of torsion  $G_{\infty}$-stable submodules of $T$
such that $T_i/T_{i-1}$ is irreducible for any $i$.
As  explained in the proof of Proposition \ref{FFLEM},
the $G_{\infty}$-action on $T_i/T_{i-1}$ can be 
(uniquely) extended to $G_K$.
By Theorem \ref{tamelift} given in the next section,
we know that 
$T_i/T_{i-1}$ is an object of 
$\mrm{Rep}^{r_0,\mrm{cris}}_{\mrm{tor}}(G_K)$
where $r_0:=\mrm{max}\{ r'\in \mbb{Z}_{\ge 0}; e(r'-1)<p-1 \}$.

\noindent 
(1) We may suppose $r=r_0$.
The $G_{\infty}$-equivariant projection $T=T_m\twoheadrightarrow T_m/T_{m-1}$ is $G_K$-stable by 
the full faithfulness theorem (= Corollary \ref{FFTHMtorcris}).
Thus we know that $T_{m-1}$ is $G_K$-stable in $T$, 
and also know that $T_{m-1}$ is contained in $\mrm{Rep}^{r,\mrm{cris}}_{\mrm{tor}}(G_K)$  
by Lemma \ref{stability'}.
By the same argument for the $G_{\infty}$-equivariant 
projection $T_{m-1}\twoheadrightarrow T_{m-1}/T_{m-2}$,
we know that 
$T_{m-2}$ is $G_K$-stable in $T$, 
and also know that $T_{m-2}$ is contained in $\mrm{Rep}^{r,\mrm{cris}}_{\mrm{tor}}(G_K)$.
Repeating this argument,
we have that $T'=T_0$ is $G_K$-stable in $T$.

\noindent
(2) Put $J=u^pI^{[p^{s-n+1}]}W(R)$. 
By (1) we may assume $e(r-1)\ge p-1$.
Under this assumption  we have $r\ge r_0$ and $s>n-1+\mrm{log}_p(r-(p-1)/e)\ge n-1$.
In particular,
$T|_{G_s}$ and $(T_i/T_{i-1})|_{G_s}$, for any $i$, are contained in $\wt{\mrm{Rep}}^{r,\hat{G}_s,J}_{\mrm{tor}}(G_s)$
by Proposition \ref{Main2Lem}.
First we consider the case where $k$ is algebraically closed. 
By Theorem \ref{FFTHM}, the $G_{\infty}$-equivariant projection
$T=T_m\twoheadrightarrow T_m/T_{m-1}$ is $G_s$-stable.
Thus we know that $T_{m-1}$ is $G_s$-stable in $T$, 
and also know that $T_{m-1}$ is contained in $\wt{\mrm{Rep}}^{r,\hat{G}_s,J}_{\mrm{tor}}(G_s)$  
by Corollary \ref{stability}.
By the same argument for the $G_{\infty}$-equivariant 
projection $T_{m-1}\twoheadrightarrow T_{m-1}/T_{m-2}$,
we know that 
$T_{m-2}$ is $G_s$-stable in $T$, 
and also know that $T_{m-2}$ is contained in  $\wt{\mrm{Rep}}^{r,\hat{G}_s,J}_{\mrm{tor}}(G_s)$.
Repeating this argument,
we have that $T'=T_0$ is $G_s$-stable in $T$.
Next we consider the case where $k$ is not necessary algebraically closed.
Let $L$ be the completion of the 
maximal unramified extension $K^{\mrm{ur}}$ of $K$, and 
we identify $G_L$ with the inertia subgroup of $G_K$.
Clearly $T|_{G_L}$ is contained in $\mrm{Rep}^{r,\mrm{cris}}_{\mrm{tor}}(G_L)$ 
and $T'$ is $G_{L_{\infty}}$-stable submodule of $T$.
We have already shown that $T'$ is $G_{L,s}$-stable in $T$.
Since $G_s$ is topologically generated by $G_{L,s}$ and $G_{\infty}$,
we conclude that $T'$ is $G_s$-stable in $T$.
\end{proof}


Now let $V$ be a $\mbb{Q}_p$-representation of $G_K$
and $T$ a $\mbb{Z}_p$-lattice of $V$ which is stable under $G_{\infty}$.
Then we know that $T$ is automatically $G_s$-stable for some $s\ge 0$.
Indeed we can check this as follows.
Take any $G_K$-stable $\mbb{Z}_p$-lattice $T'$ of $V$ which contains $T$,
and take an integer $n>0$ with the property that $p^nT'\subset T$.
Furthermore, we take a finite extension $K'$ of $K$ 
such that $G_{K'}$ acts trivially on $T'/p^nT'$.
Then $T/p^nT'$ is $G_{\infty}$-stable and also $G_{K'}$-stable in $T'/p^nT'$.
If we take any integer $s\ge 0$ with the property $K'\cap K_{\infty}\subset K_{(s)}$,
we know that $T/p^nT'$ is $G_s$-stable.
This implies that $T$ is $G_s$-stable in $T'$.

The following corollary, which was shown in Corollary 3.4.4 of \cite{Br3}
in the case where $r=1$,
is related with the above property.
\begin{corollary}
\label{stablelattice}
Let $V$ be a crystalline $\mbb{Q}_p$-representation of $G_K$
with Hodge-Tate weights in $[0,r]$ and $T$ a finitely generated $\mbb{Z}_p$-submodule
of $V$ which is stable under $G_{\infty}$.
If $e(r-1)<p-1$, then $T$ is stable under $G_K$.
\end{corollary}  

\begin{proof}
We completely follow the method of the proof of \cite[Corollary 3.4.4]{Br3}. 
Take any $G_K$-stable $\mbb{Z}_p$-lattice $T'$ of $V$ which contains $T$.
Since $T'/p^nT'$ is contained in $\mrm{Rep}^{r,\mrm{cris}}_{\mrm{tor}}(G_K)$ for any $n>0$,
Corollary \ref{imagestable} (1) implies that any $G_{\infty}$-stable submodule of $T'/p^nT'$
is in fact $G_K$-stable.
Thus $(T+p^nT')/p^nT'$ is $G_K$-stable in $T'/p^nT'$.
Therefore, we obtain $g(T)\subset \bigcap_{n>0}\ (T+p^nT')=T$ for any $g\in G_K$.    
\end{proof}


\section{Crystalline lifts and c-weights}

We continue to use the same notation except for that 
we may allow $p=2$.
We remark that a torsion $\mbb{Z}_p$-representation of $G_K$ is 
torsion crystalline with Hodge-Tate weights in $[0,r]$
if there exist a lattice $L$ in some crystalline $\mbb{Q}_p$-representation of $G_K$
with Hodge-Tate weights in $[0,r]$ and a $G_K$-equivariant surjection 
$f\colon L\twoheadrightarrow T$.  
We call $f$ a {\it crystalline lift} ({\it of $T$}) 
{\it of weight $\le r$}.
Our interest in this section is to determine the minimum integer $r$ (if it exists)
such that $T$ admits crystalline lifts of weight $\le r$.
We call this minimum integer 
the {\it c-weight of $T$} and 
denote it by $w_c(T)$. 
If $T$ does not have crystalline lifts of weight $\le r$ for any integer $r$, 
then we define the c-weight $w_c(T)$ of $T$ to be $\infty$. 
Motivated by \cite[Question 5.5]{CL2}, we raise the following question.
\begin{question}
For a torsion $\mbb{Z}_p$-representation $T$ of $G_K$,
is the c-weight $w_c(T)$ of $T$  finite?
Furthermore, 
can we calculate $w_c(T)$?
\end{question}


\subsection{General properties of c-weights}
We study general properties of c-weights.
At first, by ramification estimates,
it is known that c-weights may have 
infinitely large values (\cite[Theorem 5.4]{CL2});
for any $c>0$,
there exists a torsion $\mbb{Z}_p$-extension $T$ of $G_K$ with $w_c(T)>c$.
In this paper,
we mainly consider representations with ``small'' c-weights. 
If c-weights are ``small'',
they are closely related with {\it tame inertia weights}.
Now
we recall the definition of tame inertia weights.
Let $I_K$ be the inertia subgroup of $G_K$. 
Let $T$ be a $d$-dimensional irreducible $\mbb{F}_p$-representation of $I_K$.
Then $T$ is isomorphic to 
$$
\mbb{F}_{p^d}(\theta^{n_1}_{d,1}\cdots \theta^{n_d}_{d,d})
$$
for one sequence of integers between $0$ and $p-1$, periodic of period $d$.
Here, $\theta_{d,1},\dots , \theta_{d,d}$ are the fundamental characters of
level $d$.
The integers $n_1/e,\dots ,n_d/e$ are called the tame inertia weights of $T$.
For any $\mbb{F}_p$-representation $T$ of $G_K$,
the tame inertia weights of $T$ are the tame inertia weights of the Jordan-H\"older quotients 
of $T|_{I_K}$.

Let $\chi_p\colon G_K\to \mbb{Z}^{\times}_p$ be the $p$-adic cyclotomic character 
and $\bar{\chi}_p\colon G_K\to \mbb{F}^{\times}_p$ the mod $p$ cyclotomic character.
It is well-known that $\bar{\chi}_p|_{I_K}=\theta_1^e$
where $\theta_1\colon I_K\twoheadrightarrow \mbb{F}^{\times}_p$ is the fundamental character of level $1$.
In particular, 
denoting by $K^{\mrm{ur}}$
the maximal unramified extension of $K$,
we have $[K^{\mrm{ur}}(\mu_p):K^{\mrm{ur}}]=(p-1)/\mrm{gcd}(e,p-1)$.

\begin{proposition}
\label{poly}
$(1)$ Minimum c-weights are invariant under finite unramified extensions of the base field $K$. 

\noindent
$(2)$ The c-weight of an unramified torsion 
$\mbb{Z}_p$-representation of $G_K$ is $0$.

\noindent
$(3)$ Put $\nu=(p-1)/\mrm{gcd}(e,p-1)$. 
If $\nu (s-1)<w_c(T)\le \nu s$, then we have 
$\nu (s-1)<w_c(T^{\vee})\le \nu s$. 
In particular, if $(p-1)\mid e$, then we have $w_c(T)=w_c(T^{\vee})$.

\noindent
$(4)$ Let $T$ be an $\mbb{F}_p$-representation of $G_K$ and $i$ the largest tame inertia weight of $T$.
Then we have $w_c(T)\geq i$. 
\end{proposition}
\begin{proof}
(1)  Let $T$ be a torsion $\mbb{Z}_p$-representations of $G_K$.
Let $K'$ be a finite unramified extension of $K$.
It suffices to prove that $T$ has crystalline lifts of weight $\le r$ 
if and only if $T|_{G_{K'}}$ has crystalline lifts of weight $\le r$.
The ``only if'' assertion is clear and thus 
it is enough to prove the ``if'' assertion.
Let $f\colon L\twoheadrightarrow T|_{G_{K'}}$ be a crystalline
lift of $T|_{G_{K'}}$ of weight $\le r$.
Since $K'/K$ is unramified,
$\mrm{Ind}^{G_K}_{G_{K'}}L$ is a lattice in some crystalline 
$\mbb{Q}_p$-representation of $G_K$ with Hodge-Tate weights in $[0,r]$.
Furthermore, the map
$$
\mrm{Ind}^{G_K}_{G_{K'}}L=\mbb{Z}_p[G_K]\otimes_{\mbb{Z}_p[G_{K'}]} L\to T,\quad
\sigma\otimes x\mapsto \sigma(f(x))
$$
is a $G_K$-equivariant surjection and hence we have done.

\noindent
(2) The result follows from (1) immediately.

\noindent
(3) Taking a finite unramified extension $K'$ of $K$ with the property 
$[K^{\mrm{ur}}(\mu_p):K^{\mrm{ur}}]=[K'(\mu_p):K']$,
it follows from Lemma \ref{stability'}
that we have $\nu (s-1)<w_c(T|_{G_K'})\le \nu s$ if and only if  
we have $\nu (s-1)<w_c((T^{\vee})|_{G_K'})\le \nu s$.
Thus the result follows from the assertion (1).

\noindent
(4)
If $ew_c(T)\ge p-1$, then there is nothing to prove, and thus
we may suppose that $ew_c(T)< p-1$.
Let $L\twoheadrightarrow T$ be a crystalline lift of $T$ of weight $\le w_c(T)$.
Since the tame inertia polygon of $L$ lies on the Hodge polygon of $L$
(\cite[Th\'eor\`eme 1]{CS}), 
the largest slope of the former polygon is less than or equal to that of the latter polygon.
This implies $w_c(T)\geq i$.
\end{proof}

\begin{theorem}
\label{tamelift}
Let $T$ be a tamely ramified $\mbb{F}_p$-representation of $G_K$.
Let $i$ be the largest tame inertia weight of $T$.
Then we have $w_c(T)=\mrm{min}\{ h\in \mbb{Z}_{\ge 0}; h\ge i \}$. 
\end{theorem}

\begin{proof}
The proof below is essentially due to Caruso and Liu \cite[Theorem 5.7]{CL2},
but we give a proof here for the sake of completeness. 
Put $i_0=\mrm{min}\{ h\in \mbb{Z}_{\ge 0}; h\ge i \}$.
By Proposition \ref{poly} (4),
we have $w_c(T)\ge i_0$.
Thus it suffices to show $w_c(T)\le i_0$.
We note that $T|_{I_K}$ is semi-simple.
Any irreducible component $T_0$ of $T|_{I_K}$ is of the form 
$
\mbb{F}_{p^d}(\theta^{n_1}_{d,1}\cdots \theta^{n_d}_{d,d})
$
for one sequence of integers between $0$ and $p-1$, periodic of period $d$.
We decompose $n_j=em_j+n'_j$ by integers $0\le m_j\le i_0$ and $0\le n'_j<e$.
Now we define an integer $k_{j, \ell}$ by
\begin{align*}
k_{j, \ell}:=
\left\{
\begin{array}{ll}
e\quad \hspace{2.0mm} {\rm if}\ 1\le \ell\le m_j,  
\cr
n'_j\quad {\rm if}\ \ell=m_j+1, 
\cr
0\quad \hspace{2.0mm} {\rm if}\ \ell>m_j+1. 
\end{array}
\right.
\end{align*}
Note that we have $n_j=\sum^{i_0}_{\ell=1} k_{j, \ell}$, and
also have an $I_K$-equivariant surjection
$$
T_0=\mbb{F}_{p^d}(\theta^{n_1}_{d,1}\dots \theta^{n_d}_{d,d})
=\bigotimes_{\ell=1,\dots i_0, \mbb{F}_{p^d}} 
\mbb{F}_{p^d}(\theta^{k_{1,\ell}}_{d,1}\dots \theta^{k_{d,\ell}}_{d,d})
\twoheadleftarrow \bigotimes_{\ell=1,\dots i_0, \mbb{F}_p} 
\mbb{F}_{p^d}(\theta^{k_{1,\ell}}_{d,1}\dots \theta^{k_{d,\ell}}_{d,d}).
$$
By a classical result of Raynaud,
each $\mbb{F}_{p^d}(\theta^{k_{1,\ell}}_{d,1}\cdots \theta^{k_{d,\ell}}_{d,d})$
comes from a finite flat group scheme defined over $K^{\mrm{ur}}$.
We should remark that such a finite flat group scheme is 
in fact defined over a finite unramified extension of $K$.
Since any finite flat group scheme
can be embedded in a $p$-divisible group,
the above observation implies the following:
there exist a finite unramified extension $K'$ over $K$,
a lattice $L$ in some crystalline $\mbb{Q}_p$-representation of
$G_{K'}$ with Hodge-Tate weights in $[0, i_0]$ 
and an $I_K$-equivariant surjection $f\colon L\twoheadrightarrow T$.
The map $f$ induces an $I_K$-equivariant surjection $\tilde{f}\colon L/pL\twoheadrightarrow T$. 
Since $L/pL$ and $T$ is finite,
we see that $\tilde{f}$ is in fact $G_{K''}$-equivariant for some 
finite unramified extension $K''$ over $K'$, and then so is $f$. 
Therefore, we obtain $w_c(T|_{G_{K''}})\le i_0$.
By Proposition \ref{poly} (1), we obtain $w_c(T)\le i_0$. 
\end{proof}


\subsection{Rank $2$ cases}
We give some computations of c-weights related with torsion representations of rank $2$.
We prove the following lemma by an almost identical method with 
\cite[Lemma 9.4]{GLS}.

\begin{lemma}
\label{2liftlem}
Let $K$ be a finite extension of $\mbb{Q}_p$.
Let $E$ be a finite extension of $\mbb{Q}_p$ with residue field $\mbb{F}$.
Let $i$ and $\nu$ be integers such that
$\nu$ is divisible by $[K(\mu_p):K]$.
Suppose that $T$ is an $\mbb{F}$-representation of $G_K$
which sits in an exact sequence
$(\ast)\colon 0\to \mbb{F}(i)\to T\to \mbb{F}\to 0$
of $\mbb{F}$-representations
of $G_K$.
Then there exist a ramified degree at most $2$ extension
$E'$ over $E$, with integer ring $\cO_{E'}$,
and an unramified continuous character $\chi\colon G_K\to \mbb{F}^{\times}$
with trivial reduction
such that $(\ast)$ is  the reduction of some exact sequence
$0\to \cO_{E'}(\chi \chi_p^{i+\nu})\to \Lambda\to \cO_{E'}\to 0$
of free $\cO_{E'}$-representations of $G_K$.
Furthermore, we have the followings:

\noindent
$(1)$ If $i+\nu=1$ or $\bar{\chi}_p^{1-i}\not= 1$, then we can take $E'=E$ and $\chi=1$.

\noindent
$(2)$ If $i+\nu=0$ and $T$ is unramified,
then we can take $E'=E$, $\chi=1$ and $\Lambda$ to be unramified.  
\end{lemma}

\begin{proof}
Suppose $i+\nu=1$ (resp.\ $\bar{\chi}_p^{1-i}\not= 1$).
Then the map $H^1(K, \cO_E(i+\nu))\to H^1(K,\mbb{F}(i))$
arising from the exact sequence 
$0\to \cO_E(i+\nu)\overset{\varpi}{\to} \cO_E(i+\nu)\to \mbb{F}(i)\to 0$
is surjective since 
$H^2(K, \cO_E(1))\simeq \cO_E$ (resp.\ $H^2(K, \cO_E(i+\nu))=0$),
where $\varpi$ is a uniformizer of $E$.
Hence we obtained a proof of (1).
The assertion (2)  follows immediately from the fact that the natural map
$H^1(G_K/I_K, \cO_E)\to H^1(G_K/I_K, \mbb{F})$ is surjective.

In the rest of this proof, we always assume that $i+\nu\not=1$ and $\bar{\chi}_p^{1-i}= 1$.
Let $L\in H^1(K,\mbb{F}(i))$ be 
a $1$-cocycle corresponding to 
$(\ast)$. 
We may suppose $L \not=0$.
For any unramified continuous character $\chi\colon G_K\to \mbb{F}^{\times}$
with trivial reduction,
we denote by 
\begin{align*}
&\delta^1_{\chi}\colon H^1(K,\mbb{F}(i))\to H^2(K, \cO_E(\chi\chi_p^{i+\nu}))\\
(\mrm{resp.}\ &\delta^0_{\chi}\colon 
H^0(K, E/\cO_E(\chi^{-1}\chi_p^{1-i-\nu}))\to H^1(K, \mbb{F}))
\end{align*}
the connection map arising from the exact sequence 
$0\to \cO_E(\chi\chi_p^{i+\nu})\overset{\varpi}{\to} 
\cO_E(\chi\chi_p^{i+\nu}) \to \mbb{F}(i)\to 0$
(resp.\ 
$0\to \mbb{F}\to E/\cO_E(\chi^{-1}\chi_p^{1-i-\nu})\overset{\varpi}{\to} 
E/\cO_E(\chi^{-1}\chi_p^{1-i-\nu})\to 0$) 
of $\cO_E[G_K]$-modules.
Consider the following commutative diagram:
\begin{center}
$\displaystyle \xymatrix{
H^1(K,\mbb{F}(i)) \ar_{\delta^1_{\chi}}[d] 
& \times 
& H^1(K,\mbb{F}) \ar[rrr] 
& &
& E/\cO_E \ar@{=}[d] \\
H^2(K, \cO_E(\chi\chi_p^{i+\nu})) 
& \times 
& H^0(K, E/\cO_E(\chi^{-1}\chi_p^{1-i-\nu})) \ar^{\delta^0_{\chi}}[u] \ar[rrr] 
& &
& E/\cO_E
}$
\end{center}
Since the above two pairings are perfect,
we see that $L$ lifts to
$H^1(G_K, \cO_E(\chi\chi_p^{i+\nu}))$
if and only if $H$ is contained in the image of $\delta^0_{\chi}$.
Here, $H\subset H^1(K,\mbb{F})$ is the annihilator of $L$
under the local Tate pairing 
$H^1(K,\mbb{F}(i)) \times H^1(K,\mbb{F}) \to E/\cO_E$.
Let  $n\ge 1$ be the largest integer with the property that
$\chi^{-1}\chi_p^{1-i-\nu}\equiv 1\ \mrm{mod}\ \varpi^n$
(such $n$ exists since $\bar{\chi}_p^{1-i}=1$ and $1-i-\nu\not=0$). 
We define $\alpha_{\chi}\colon G_K\to \cO_E$ by the relation
$\chi^{-1}\chi_p^{1-i-\nu}=1+\varpi^n\alpha_{\chi}$, and denote 
$(\alpha_{\chi}\ \mrm{mod}\ \varpi)\colon G_K\to \mbb{F}$ by
$\bar{\alpha}_{\chi}$. 
By definition, $\bar{\alpha}_{\chi}$ 
is a non-zero element of $H^1(K,\mbb{F})$,
and it is not difficult to check that 
the image of $\delta^0_{\chi}$ is generated by $\bar{\alpha}_{\chi}$.
If $\bar{\alpha}_{\chi}$ is contained in $H$ for some $\chi$,
we are done. Suppose this is not the case.
 
Suppose that $H$ is not contained in the unramified line in 
$H^1(K,\mbb{F})$.
We claim that we can choose $\chi$ such that $\bar{\alpha}_{\chi}$
is ramified.
Let $m$ be the largest integer 
with the property that
$(\chi^{-1}\chi_p^{1-i-\nu})|_{I_K}\equiv 1\ \mrm{mod}\ \varpi^n$.
Clearly, we have $m\ge n$.
If $m=n$, then we are done and thus we may assume $m>n$.
Fix a lift $g\in G_K$ of the Frobenius of $K$.
We see that $\bar{\alpha}_{\chi}(g)\not= 0$. 
Let $\chi'$ be the unramified character sending $g$ to $1+\varpi^n\alpha_{\chi}(g)$.
Then $\chi'$ has trivial reduction.
After replacing $\chi$ with $\chi\chi'$,
we reduce the case where $m=n$ and thus the claim follows.
Suppose $\bar{\alpha}_{\chi}$
is ramified. 
Then there exists a unique $\bar{x}\in \mbb{F}^{\times}$
such that $\bar{\alpha}_{\chi}+u_{\bar x}\in H$ where 
$u_{\bar x}\colon G_K\to \mbb{F}$ is the unramified character 
sending $g$ to $\bar x$.
Denote by $\chi''$ the unramified character sending 
$g$ to $1+\varpi^n\alpha_{\chi}(g)$.
Replacing $\chi$ with  $\chi\chi''$, we have done.

Suppose that $H$ is contained in the unramified line in 
$H^1(K,\mbb{F})$ (thus $H$ and the unramified line  coincide with each other).
By replacing $E$ with $E(\sqrt{\varpi})$,
we may assume that $n>1$.
Let $\chi_0$ be a character defined by $\chi$ times the unramified character sending our fixed 
$g$ to $1+\varpi$.
Since $n>1$, we see that $\chi_0^{-1}\chi_p^{1-i-\nu}\equiv 1\ \mrm{mod}\ \varpi$
and $\chi_0^{-1}\chi_p^{1-i-\nu}\not\equiv 1\ \mrm{mod}\ \varpi^2$.
We define $\alpha_{\chi_0}\colon G_K\to \cO_E$ by the relation
$\chi_0^{-1}\chi_p^{1-i-\nu}=1+\varpi\alpha_{\chi_0}$, and denote 
$(\alpha_{\chi_0}\ \mrm{mod}\ \varpi)\colon G_K\to \mbb{F}$ by
$\bar{\alpha}_{\chi_0}$. 
By definition and the assumption $n>1$, 
$\bar{\alpha}_{\chi_0}$ 
is a non-zero unramified element of $H^1(K,\mbb{F})$,
hence it is contained in $H$. 
Therefore, we have done.
\end{proof}

Let $K$ be a finite extension of $\mbb{Q}_p$, 
$n\ge 2$ an integer and 
$\chi\colon G_K\to E^{\times}$ an unramified character.
Since any $E$-representation of $G_K$
which is an extension of $E$ by $E(\chi\chi_p^n)$
is automatically crystalline,
we obtain the following.

\begin{proposition}
\label{rank2}
Suppose $p>2$.
Let $K$ be a finite unramified extension of $\mbb{Q}_p$.
Let $T\in \mrm{Rep}_{\mrm{tor}}(G_K)$ be killed by $p$ and  
sit in an exact sequence
$0\to \mbb{F}_p(i)\to T\to \mbb{F}_p\to 0$ of $\mbb{F}_p$-representations
of $G_K$. Then we have the followings:

\noindent
$(1)$ If $i=0$ and $T$ is unramified, 
then we have
$w_c(T)=0$.

\noindent
$(2)$ If $i=0$ and $T$ is not unramified, 
then we have
$w_c(T)=p-1$.

\noindent
$(3)$ If $i=2,\dots, p-2$, 
then we have
$w_c(T)=i$. 
\end{proposition}

\begin{proof}
(1), (2)
By Lemma \ref{2liftlem} (2),
it suffices to prove  that $T$ is not torsion crystalline with Hodge-Tate weights in $[0,p-2]$
if $T$ is not unramified. 
Let $K_T$ be the definition field of the representation $T$ of $G_K$
and put $G=\mrm{Gal}(K_T/K)$.
Let $G^j$ be the upper numbering $j$-th ramification subgroup of $G$ (in the sense of \cite{Se}). 
Since $T$ is not unramified and killed by $p$,
we see that $K_T$ is a totally ramified degree $p$ extension over $K$.
Thus $G^1$ is the wild inertia subgroup of $G$ and $G^1=G$, 
which does not act on $T$ trivial by the definition of $G$.
Thus we obtain the desired result by ramification estimates of \cite{Fo1}
(or \cite{Ab1}) for torsion crystalline representations with Hodge-Tate weights in $[0,p-2]$:
if $T$ is torsion crystalline with Hodge-Tate weights in $[0,p-2]$,
then $G^j$ acts on $T$ trivial for any $j>(p-2)/(p-1)$.

\noindent
(3) The result follows immediately from Proposition \ref{poly} (4) and Lemma \ref{2liftlem}.
\end{proof}

\begin{corollary}
\label{2lift}
Let $K$ be a finite unramified extension of $\mbb{Q}_p$.
Then any $2$-dimensional $\mbb{F}_p$-representation of $G_K$
is torsion crystalline with Hodge-Tate weights in $[0,2p-2]$.
\end{corollary}

\begin{proof}
If $T$ is irreducible,
the result follows from Theorem \ref{tamelift}.
Assume that $T$ is reducible.
Since $K$ is unramified over $\mbb{Q}_p$,
any continuous character $G_K\to \mbb{F}^{\times}_p$
is of the form $\chi \bar{\chi}^i_p$ for some unramified character $\chi$ and some integer $i$.
Replacing $K$ with its finite unramified extension, 
we may assume that $T$ sits in an exact sequence
$0\to \mbb{F}_p(i)\to T\to \mbb{F}_p(j)\to 0$ 
of $\mbb{F}_p$-representations
of $G_K$,
where $i$ and $j$ are integers in the range $[0,p-2]$
(we remark that $w_c(T)$ is invariant under unramified extensions of $K$
by Proposition \ref{poly} (1)).
It follows from Lemma \ref{2liftlem}
that $w_c(T(-j))\le p$.
Therefore, we obtain 
$w_c(T)=w_c(T(-j)\otimes_{\mbb{F}_p} \mbb{F}_p(j))
\le w_c(T(-j))+ w_c(\mbb{F}_p(j))\le p+(p-2)
=2p-2$.
\end{proof}


\subsection{Extensions of $\mbb{F}_p$ by $\mbb{F}_p(1)$ and non-fullness theorems}
By Lemma \ref{2liftlem},
we know that the c-weight $w_c(T)$ of an $\mbb{F}_p$-representation $T$ of $G_K$
which sits in an exact sequence
$0\to \mbb{F}_p(1)\to T\to \mbb{F}_p\to 0$ of $\mbb{F}_p$-representations
of $G_K$, is less than or equal to $p$.
Let us calculate $w_c(T)$ for such $T$ more precisely.
We should remark that such $T$ is written as $p$-torsion points of a Tate curve.
Hence we consider  torsion representations coming from Tate curves. 

Let $v_K$ be the valuation of $K$ normalized such that $v_K(K^{\times})=\mbb{Z}$,
and take any $q\in K^{\times}$ with $v_K(q)>0$.
Let $E_q$ be the Tate curve over $K$ 
associated with $q$ and $E_q[p^n]$ the module of 
$p^n$-torsion points of $E_q$ for any integer $n>0$.
It is well-known that there exists an exact sequence 
$$
(\#)\quad  0\to \mu_{p^n}\to E_q[p^n]\to \mbb{Z}/p^n\mbb{Z}\to 0
$$
of $\mbb{Z}_p[G_K]$-modules. 
Here, $\mu_{p^n}$ is the group of $p^n$-th roots of unity in $\overline{K}$.
Let $x_n\colon G_K\to \mu_{p^n}$ be the $1$-cocycle defined to be the 
image of $1$ for the connection map 
$H^0(K, \mbb{Z}/p^n\mbb{Z})\to H^1(K, \mu_{p^n})$
arising from the exact sequence $(\#)$.
Then $x_n$ corresponds to $q$ mod $(K^{\times})^{p^n}$ via the isomorphism 
$K^{\times}/(K^{\times})^{p^n}\simeq H^1(K,\mu_{p^n})$
of Kummer theory.
Thus the exact sequence $(\#)$ splits if and only if $q\in (K^{\times})^{p^n}$.

First we consider the case $p\mid v_K(q)$ (i.e.\  {\it peu ramifi\'e} case).

\begin{lemma}
\label{easycase}
Let $K$ be a finite extension of $\mbb{Q}_p$.
If $p\mid v_K(q)$,
then $E_q[p]$ is the reduction modulo $p$ of a lattice in some
$2$-dimensional crystalline $\mbb{Q}_p$-representation 
with Hodge-Tate weights in $[0,1]$.
\end{lemma}

\begin{proof}
Since $p\mid v_K(q)$,
there exists $q'\in K^{\times}$ such that $v_K(q'-1)>0$ and $q\equiv q'$ mod $(K^{\times})^p$.
Considering the exact sequence $0\to \mbb{Z}_p(1)\to L\to \mbb{Z}\to 0$ of $\mbb{Z}_p$-representations of 
$G_K$ corresponding to $q'$ via the isomorphism 
$H^1(K,\mbb{Z}_p(1))\simeq \plim_{n} K^{\times}/(K^{\times})^{p^n}$ 
of Kummer theory,
we obtain the desired result. 
\end{proof}

\begin{corollary}
\label{minwt}
Suppose that $K$ is a finite extension of $\mbb{Q}_p$,
$(p-1)\nmid e$ and $p\mid v_K(q)$.
Then we have $w_c(E_q[p])=1$.
\end{corollary}

\begin{proof}
By the assumption $(p-1)\nmid e$,
we know that the largest tame inertia weight of $E_q[p]$ is positive.
Thus Proposition \ref{poly} (4) shows $w_c(E_q[p])\ge 1$. 
The inequality $w_c(E_q[p])\le 1$ follows from
Lemma \ref{easycase}.
\end{proof}

Next we consider the case $p\nmid v_K(q)$ (i.e.\ {\it tr\`es ramifi\'e} case).

\begin{proposition}
\label{Tatecurve}
If $e(r-1)<p-1$ and $p\nmid v_K(q)$,
then $E_q[p^n]$ is not 
torsion crystalline with Hodge-Tate weights in $[0,r]$ for any $n>0$.
\end{proposition}

\begin{remark}
If $e=1$, the fact that $E_{\pi}[p^n]$ is not torsion crystalline with Hodge-Tate weights in $[0,p-1]$ 
immediately follows from the theory of ramification bound as below.
We may suppose $n=1$.
Suppose $E_{\pi}[p]$ is 
torsion crystalline with Hodge-Tate weights in $[0,p-1]$.
Then the upper numbering $j$-th ramification subgroup $G^j_K$ of 
$G_K$ (in the sense of \cite{Se}) 
acts trivially 
on $E_{\pi}[p]$ for any $j>1$ (\cite[Section 6, Theorem 3.1]{Ab1}).
However, this contradicts the fact that  
the upper bound of the ramification of $E_{\pi}[p]$
is $1+1/(p-1)$.
\end{remark}

\begin{proof}[Proof of Proposition \ref{Tatecurve}]
We may suppose $n=1$. We choose any uniformizer $\pi'$ of $K$.
Putting $v_K(q)=m$,
we can write $q=(\pi')^mx$ with some unit $x$ of the integer ring of $K$.
Since $m$ is prime to $p$,
we have a decomposition $x=\zeta_{\ell}y^m$ in $K^{\times}$
for some $\ell>0$ prime to $p$ and $y\in K$ with $v_K(y-1)>0$.
Here $\zeta_{\ell}$ is a (not necessary primitive) $\ell$-th root of unity.
Since $\ell$ is prime to $p$, we have $\zeta_{\ell}=\zeta^{ps}_{\ell}$
for some integer $s$.
We put $\pi=\pi'y$. This is a uniformizer of $K$.
Choose any $p$-th root $\pi_1$ of  $\pi$ and put $q_1=\zeta^s_{\ell}\pi^m_1\in K(\pi_1)^{\times}$.
Then we have $q=q^p_1\in (K(\pi_1)^{\times})^p$ and in particular,
the exact sequence $(\#)$ (for $n=1$)
splits as representations of $\mrm{Gal}(\overline{K}/K(\pi_1))$.
Now assume that $E_q[p]$ is 
torsion crystalline with Hodge-Tate weights in $[0,r]$.
Then $(\#)$ (for $n=1$) splits as representations of $G_K$ by 
Corollary \ref{FFTHMtorcris}. This contradicts the assumption $p\nmid v_K(q)$ 
(and hence $q\notin (K^{\times})^p$).
\end{proof}

Now we put $r'_0=\mrm{min}\{r\in \mbb{Z}_{\ge 0} ; e(r-1)\ge p-1\}$.
Recall that we have  $[K^{\mrm{ur}}(\mu_p):K^{\mrm{ur}}]=(p-1)/\mrm{gcd}(e,p-1)$.

\begin{lemma}
\label{roughbound}
Let $K$ be a finite extension of $\mbb{Q}_p$.
Then $E_q[p]$ is torsion crystalline  
with Hodge-Tate weights in $[0,1+(p-1)/\mrm{gcd}(e,p-1)]$.
\end{lemma}

\begin{proof}
Taking a finite unramified extension $K'$ of $K$ such that
$[K^{\mrm{ur}}(\mu_p):K^{\mrm{ur}}]=[K'(\mu_p):K']$,
we obtain $w_c((E_q[p])|_{G_{K'}})\le 1+(p-1)/\mrm{gcd}(e,p-1)$ 
by Lemma \ref{2liftlem}.
Thus we have $w_c(E_q[p])\le 1+(p-1)/\mrm{gcd}(e,p-1)$ 
by Proposition \ref{poly} (1).
\end{proof}

\begin{corollary}
\label{wctr1}
Suppose that $K$ is a finite extension of $\mbb{Q}_p$,
and also suppose $e\mid (p-1)$ or $(p-1)\mid e$.
We further suppose that $p\nmid v_K(q)$.
Then we have $w_c(E_q[p])=r'_0$. 
\end{corollary}

\begin{proof}
We have $w_c(E_q[p])\le r'_0$ by Lemma \ref{roughbound}.
In addition, we also have $w_c(E_q[p])\ge r'_0$ by Proposition \ref{Tatecurve}.
\end{proof}

Lemma \ref{roughbound} gives some non-fullness results on 
torsion crystalline representations.

\begin{corollary}
\label{nonfullthm}
Suppose that $K$ is a finite extension of $\mbb{Q}_p$.
If $r\ge 1+(p-1)/\mrm{gcd}(e,p-1)$,
then the restriction functor 
$\mrm{Rep}^{r, \mrm{cris}}_{\mrm{tor}}(G_K)\to 
\mrm{Rep}_{\mrm{tor}}(G_1)$ 
is not full.
\end{corollary}

\begin{proof}
Two representations $E_{\pi}[p]$ and $\mbb{F}_p(1)\oplus \mbb{F}_p$ are
objects of $\mrm{Rep}^r_{\mrm{tor}}(G_K)$ 
by Lemma \ref{roughbound}.
They are not isomorphic as representations of 
$G_K$ but isomorphic as representations of $G_1$.  
Thus the desired non-fullness follows.
\end{proof}
 
\begin{corollary}
\label{p2}
Suppose that any one of the following holds:
\begin{itemize}
\item $p=2$ and $K$ is a finite extension of $\mbb{Q}_2$ $($in this case $r'_0=2)$;
\item $K$ is a finite unramified extension of $\mbb{Q}_p$ $($in this case $r'_0=p)$;
\item $K$ is a finite extension of $\mbb{Q}_p(\mu_p)$ $($in this case $r'_0=2)$.
\end{itemize}
Then the restriction functor 
$\mrm{Rep}^{r, \mrm{cris}}_{\mrm{tor}}(G_K)\to 
\mrm{Rep}_{\mrm{tor}}(G_1)$ 
is not full.
\end{corollary}


\begin{thebibliography}{1000}

\bibitem[Ab1]{Ab1}
Victor Abrashkin,
\emph{Modular representations of the Galois group of a local field
      and a generalization of a conjecture of Shafarevich},
      Izv. Akad. Nauk SSSR Ser. Mat. {\bf 53} (1989), no. 6, 
      1135--1182 (Russian),      
      Math. USSR-Izv. {\bf 35}, no. 3,
      469--518 (English).  

      
\bibitem[Ab2]{Ab2}
Victor Abrashkin,
\emph{Group schemes of period $p>2$},
      Proc. Lond. Math. Soc. (3), {\bf 101} (2010),
      207--259.      

\bibitem[Br1]{Br1}
Christophe Breuil, 
\emph{Repr\'esentations {$p$}-adiques semi-stables et
      transversalit\'e de Griffiths}, 
      Math. Ann. {\bf 307}  (1997),
      191--224. 
      
\bibitem[Br2]{Br2}
Christophe Breuil, 
\emph{Une application du corps des normes},
      Compos. Math. {\bf 117} (1999) 
      189--203.
      
\bibitem[Br3]{Br3}
Christophe Breuil, 
\emph{Integral $p$-adic Hodge theory},
      in Algebraic geometry 2000, Azumino (Hotaka), 
      Adv. Stud. Pure Math., vol. 36, Math. Soc. Japan, 2002,
      51--80.



\bibitem[Ca]{Ca}
Xavier Caruso, 
{\rm Repr\'esentations semi-stables de torsion dans le case $er < p-1$},
J.\ Reine Angew.\ Math.\ {\bf 594} (2006), 35--92.

\bibitem[CL1]{CL1}
Xavier Caruso and Tong Liu,
\emph{Quasi-semi-stable representations},
      Bull. Soc. Math. France, {\bf 137} (2009), no. 2,
      185--223.
      
\bibitem[CL2]{CL2}
Xavier Caruso and Tong Liu,
\emph{Some bounds for ramification of $p^n$-torsion
      semi-stable representations},      
      J.  Algebra, {\bf 325}  (2011),
      70--96. 
      
\bibitem[CS]{CS}
Xavier Caruso and David Savitt, 
\emph{Polygons de Hodge, de Newton et de 
      l\'inertie moderee des representations semi-stables}, 
      Math. Ann. {\bf 343} (2009):
      773--789.

\bibitem[Fo1]{Fo1}
Jean-Marc Fontaine, 
\emph{Sc\'emas propres et lisses sur $\mbb{Z}$},
      Proceedings of the Indo-French Conference on Geometry
      Bombay, 1989, (Hindustan Book Agency 1993),
      43--56.


\bibitem[Fo2]{Fo2}
Jean-Marc Fontaine, 
\emph{Le corps des p\'eriodes $p$-adiques},
      Ast\'erisque (1994), no. 223, 
      59--111, With an appendix by Pierre Colmez,
      P\'eriodes $p$-adiques (Bures-sur-Yvette, 1988).
          

\bibitem[GLS]{GLS}
Toby Gee, Tong Liu and David Savitt, 
\emph{The Buzzard-Diamond-Jarvis conjecture for unitary groups},
      J. Amer. Math. Soc. {\bf 27} (2014), 389-435.


\bibitem[Kis]{Kis}
Mark Kisin, 
\emph{Crystalline representations and {$F$}-crystals},
Algebraic geometry and number theory, 
Progr. Math. {\bf 253},
Birkh\"auser Boston,
Boston, MA  (2006),
459--496.
      
\bibitem[Kim]{Kim}
Wansu Kim, 
\emph{The classification of $p$-divisible groups over 2-adic discrete valuation rings}, 
      Math. Res. Lett. {\bf 19} (2012), no. 1,
      121--141.

\bibitem[La]{La}
Eike Lau,
\emph{A relation between Dieudonn\'e displays and crystalline Dieudonn\'e theory}, 
      preprint, arXiv:1006.2720v3


\bibitem[Li1]{Li1}
Tong Liu,
\emph{Torsion $p$-adic Galois representations and 
      a conjecture of Fontaine}, 
Ann. Sci. \'Ecole Norm. Sup. (4) {\bf 40} (2007), no. 4,
633--674.


\bibitem[Li2]{Li2}
Tong Liu,
\emph{A note on lattices in semi-stable representations}, 
      Math. Ann. {\bf 346} (2010),
117--138.


\bibitem[Li3]{Li4}
Tong Liu,
\emph{Lattices in filtered $(\vphi,N)$-modules}, 
J. Inst. Math. Jussieu 2, Volume 11, Issue 03 (2012), 659-693.

\bibitem[Li4]{Li3}
Tong Liu,
\emph{The correspondence between Barsotti-Tate groups and Kisin modules when $p=2$}, 
appear at Journal de Th\'eroie des Nombres de Bordeaux.

\bibitem[Qu]{Qu}
Daniel Quillen,
\emph{Higher algebraic $K$-theory: I},
      in Algebraic $K$-theory, I: Higher $K$-theories (Seattle, 1972), 
      Lecture Notes in Math. {\bf 341}, Springer-Verlag, New York, 
      1973, 85--147.


\bibitem[Se]{Se}
Jean-Pierre Serre,
\emph{Corps locaux}, 
      Hermann, Paris, 1968.



\bibitem[Oz1]{Oz1}
Yoshiyasu Ozeki,
\emph{Torsion representations arising from $(\vphi,\hat{G})$-modules}, 
      J. Number Theory {\bf 133} (2013), 3810--3861.
      
\bibitem[Oz2]{Oz2}
Yoshiyasu Ozeki,
\emph{Full faithfulness theorem for torsion crystalline representations}, 
      preprint, arXiv:1206.4751v4





\end{thebibliography}
\end{document}